\documentclass[11pt,english]{article}
\usepackage[T1]{fontenc}
\usepackage[latin9]{inputenc}
\usepackage{verbatim}
\usepackage{float}
\usepackage{amsmath}
\usepackage{amsthm}
\usepackage{amssymb}
\usepackage{graphicx}

\makeatletter
\theoremstyle{plain}
\newtheorem{thm}{\protect\theoremname}
\theoremstyle{plain}
\newtheorem{question}[thm]{\protect\questionname}
\theoremstyle{plain}
\newtheorem{cor}[thm]{\protect\corollaryname}
\theoremstyle{plain}
\newtheorem{lem}[thm]{\protect\lemmaname}
\theoremstyle{remark}
\newtheorem*{rem*}{\protect\remarkname}
\theoremstyle{plain}
\newtheorem{prop}[thm]{\protect\propositionname}
\theoremstyle{plain}
\newtheorem{fact}[thm]{\protect\factname}
\theoremstyle{remark}
\newtheorem{rem}[thm]{\protect\remarkname}

\usepackage{fullpage}
\usepackage{hyperref}
\usepackage{tikz}
\usetikzlibrary{shapes.geometric, arrows, decorations.markings}
\usepackage{centernot}
\usepackage[bottom]{footmisc}
\hypersetup{
  colorlinks   = true, 
  urlcolor     = red, 
  linkcolor    = blue, 
  citecolor   = blue 
}
\date{}
\makeatother

\usepackage{babel}
\providecommand{\corollaryname}{Corollary}
\providecommand{\factname}{Fact}
\providecommand{\lemmaname}{Lemma}
\providecommand{\propositionname}{Proposition}
\providecommand{\questionname}{Question}
\providecommand{\remarkname}{Remark}
\providecommand{\theoremname}{Theorem}

\begin{document}
\global\long\def\goinf{\rightarrow\infty}%
\global\long\def\gozero{\rightarrow0}%
\global\long\def\bra{\langle}%
\global\long\def\ket{\rangle}%
\global\long\def\union{\cup}%
\global\long\def\intersect{\cap}%
\global\long\def\abs#1{\left|#1\right|}%
\global\long\def\norm#1{\left\Vert #1\right\Vert }%
\global\long\def\floor#1{\left\lfloor #1\right\rfloor }%
\global\long\def\ceil#1{\left\lceil #1\right\rceil }%
\global\long\def\expect{\mathbb{E}}%
\global\long\def\e{\mathbb{E}}%
\global\long\def\f{\mathbb{\mathbb{F}}}%
\global\long\def\r{\mathbb{R}}%
\global\long\def\n{\mathbb{N}}%
\global\long\def\q{\mathbb{Q}}%
\global\long\def\c{\mathbb{C}}%
\global\long\def\p{\mathbb{P}}%
\global\long\def\z{\mathbb{Z}}%
\global\long\def\grad{\nabla}%
\global\long\def\all{\forall}%
\global\long\def\eps{\varepsilon}%
\global\long\def\quadvar#1{V_{2}^{\pi}\left(#1\right)}%
\global\long\def\cross{\times}%
\global\long\def\del{\nabla}%
\global\long\def\parx#1{\frac{\partial#1}{\partial x}}%
\global\long\def\pary#1{\frac{\partial#1}{\partial y}}%
\global\long\def\parz#1{\frac{\partial#1}{\partial z}}%
\global\long\def\part#1{\frac{\partial#1}{\partial t}}%
\global\long\def\partheta#1{\frac{\partial#1}{\partial\theta}}%
\global\long\def\parr#1{\frac{\partial#1}{\partial r}}%
\global\long\def\curl{\nabla\times}%
\global\long\def\rotor{\nabla\times}%
\global\long\def\one{\mathbf{1}}%
\global\long\def\Hom{\mathrm{Hom}}%
\global\long\def\pr#1{\text{Pr}\left[#1\right]}%
\global\long\def\almost{\mathbf{\approx}}%
\global\long\def\tr{\text{Tr}}%
\global\long\def\var{\mathrm{Var}}%
\global\long\def\onenorm#1{\left\Vert #1\right\Vert _{1}}%
\global\long\def\twonorm#1{\left\Vert #1\right\Vert _{2}}%
\global\long\def\Inj{\mathfrak{Inj}}%
\global\long\def\inj{\mathsf{inj}}%
\global\long\def\discrete{\mathcal{C}_{n}}%
\global\long\def\contin{\overline{\mathcal{C}_{n}}}%
\global\long\def\lone{\mathrm{L}^{1}}%
\global\long\def\ltwo{\mathrm{L}^{2}}%
\global\long\def\elone{\mathrm{L}^{1}}%
\global\long\def\eltwo{\mathrm{L}^{2}}%
\global\long\def\Inf{\mathrm{Inf}}%
\global\long\def\i{\mathrm{I}}%
\global\long\def\sign{\mathrm{sign}}%
\global\long\def\tensor{\otimes}%

\title{Concentration on the Boolean hypercube via pathwise stochastic analysis}
\author{Ronen Eldan\thanks{Weizmann Institute of Science. Incumbent of the Elaine Blond Career
Development Chair. Supported by a European Research Council Starting
Grant (ERC StG) and by an Israel Science Foundation grant no. 718/19.
Email: ronen.eldan@weizmann.ac.il.} ~and Renan Gross\thanks{Weizmann Institute of Science. Supported by the Adams Fellowship Program
of the Israel Academy of Sciences and Humanities, the European Research
Council and by the Israeli Science Foundation. Email: renan.gross@weizmann.ac.il. }}
\maketitle
\begin{abstract}
We develop a new technique for proving concentration inequalities
which relate between the variance and influences of Boolean functions.
Using this technique, we 
\begin{enumerate}
\item Settle a conjecture of Talagrand \cite{talagrand_on_boundaries_and_influences},
proving that 
\[
\int_{\left\{ -1,1\right\} ^{n}}\sqrt{h_{f}\left(x\right)}d\mu\geq C\cdot\var\left(f\right)\cdot\left(\log\left(\frac{1}{\sum\Inf_{i}^{2}\left(f\right)}\right)\right)^{1/2},
\]
where $h_{f}\left(x\right)$ is the number of edges at $x$ along
which $f$ changes its value, and $\Inf_{i}\left(f\right)$ is the
influence of the $i$-th coordinate.
\item Strengthen several classical inequalities concerning the influences
of a Boolean function, showing that near-maximizers must have large
vertex boundaries. An inequality due to Talagrand states that for
a Boolean function $f$, $\var\left(f\right)\leq C\sum_{i=1}^{n}\frac{\Inf_{i}\left(f\right)}{1+\log\left(1/\Inf_{i}\left(f\right)\right)}$.
We give a lower bound for the size of the vertex boundary of functions
saturating this inequality. As a corollary, we show that for sets
that satisfy the edge-isoperimetric inequality or the Kahn-Kalai-Linial
inequality up to a constant, a constant proportion of the mass is
in the inner vertex boundary. 
\item Improve a quantitative relation between influences and noise stability
given by Keller and Kindler. 
\end{enumerate}
Our proofs rely on techniques based on stochastic calculus, and bypass
the use of hypercontractivity common to previous proofs.
\end{abstract}
\pagebreak{}

\tableofcontents{}

\section{Introduction}

\subsection{Background}

The influence of a Boolean function $f:\left\{ -1,1\right\} ^{n}\to\left\{ -1,1\right\} $
in direction $i=1,\ldots,n$ is defined as 
\[
\Inf_{i}\left(f\right)=\mu\left(\left\{ y\in\left\{ -1,1\right\} ^{n}\mid f\left(y\right)\neq f\left(y^{\oplus i}\right)\right\} \right),
\]
where $y^{\oplus i}$ is the same as $y$ but with the $i$-th bit
flipped, and $\mu$ is the uniform measure on the discrete hypercube
$\left\{ -1,1\right\} ^{n}$. The expectation and variance of a function
are given by 
\[
\e f=\int_{\left\{ -1,1\right\} ^{n}}fd\mu\text{ and }\var\left(f\right)=\e\left(f-\e f\right)^{2}.
\]
The Poincaré inequality gives an immediate relation between the aforementioned
quantities, namely,
\begin{equation}
\var\left(f\right)\leq\sum_{i=1}^{n}\Inf_{i}\left(f\right).\label{eq:poincare_inequality}
\end{equation}
 The total influence $\sum_{i}\Inf_{i}\left(f\right)$ on the right
hand side is equal to the number of edges of the hypercube which separate
$f\left(x\right)=1$ and $f\left(x\right)=-1$. It can therefore be
seen as a type of surface-area of $f$. 

The inequality (\ref{eq:poincare_inequality}) in fact holds for any
function (for a suitably defined influence), and it is natural to
ask whether it can be improved when Boolean functions are considered.
A corollary of the breakthrough paper by Kahn, Kalai and Linial (KKL)
\cite{KKL} shows that this inequality can be strengthened logarithmically:
There exists a universal constant $C>0$ such that 
\begin{equation}
\var\left(f\right)\leq C\frac{\sum_{i}\Inf_{i}\left(f\right)}{\log\left(1/\max_{i}\left(\Inf_{i}\left(f\right)\right)\right)}.\label{eq:KKL}
\end{equation}
(The above formulation does not appear explicitly in \cite{KKL},
but follows easily from their methods). 

The KKL inequality is tight for the Tribes function, but is off by
a factor of $\sqrt{n}/\log n$ for the majority function, whose influences
are all of order $1/\sqrt{n}$, suggesting that the total influence
$\sum\Inf_{i}\left(f\right)$ may not be the right notion of surface-area
for all Boolean functions. In \cite[Theorem 1.1]{talagrand_isoperimetry_log_sobolev},
Talagrand showed that 
\begin{equation}
\var\left(f\right)\leq\frac{1}{\sqrt{2}}\e\sqrt{h_{f}},\label{eq:talagrand_surface_area}
\end{equation}
where $h_{f}\left(y\right)=\#\left\{ i\in\left[n\right]\mid f\left(y\right)\neq f\left(y^{\oplus i}\right)\right\} $
is the sensitivity of $f$ at point $y$. The value $\e\sqrt{h_{f}}$
can be seen as another type of surface-area of the function $f$.
A sharp tightening of this inequality was given by Bobkov \cite{bobkov_isoperimetric_inequality_on_the_discrete_cube};
his inequality gives an elementary proof of the isoperimetric inequality
on Gaussian space. Inequality (\ref{eq:talagrand_surface_area}) is
tight for linear-threshold functions such as majority, but not for
Tribes. Thus, neither inequality implies the other. This raises the
following question:
\begin{question}
What is the right notion of boundary for Boolean functions? Is there
an inequality from which both (\ref{eq:KKL}) and (\ref{eq:talagrand_surface_area})
can be derived?
\end{question}

As a step in this direction, Talagrand conjectured in \cite{talagrand_on_boundaries_and_influences}
that (\ref{eq:talagrand_surface_area}) can be strengthened, and that
there exists a constant $\beta>0$ such that 
\begin{equation}
\e\sqrt{h_{f}}\geq\beta\cdot\var\left(f\right)\cdot\left(\log\left(\frac{e}{\sum\Inf_{i}\left(f\right)^{2}}\right)\right)^{1/2}.\label{eq:talagrand_conjecture_first_appearance}
\end{equation}
Talagrand showed that there exists an $\alpha\leq1/2$ and a constant
$\beta>0$ such that 
\[
\int_{\left\{ -1,1\right\} ^{n}}\sqrt{h_{f}\left(x\right)}d\mu\geq\beta\cdot\var\left(f\right)\left(\log\frac{e}{\var\left(f\right)}\right)^{1/2-\alpha}\cdot\left(\log\left(\frac{e}{\sum\Inf_{i}\left(f\right)^{2}}\right)\right)^{\alpha},
\]
but his proof did not yield the conjectured $\alpha=1/2$, and it
falls short of recovering the logarithmic improvement in the KKL inequality. 

Another notion of surface-area is the \emph{vertex boundary} $\partial f$
of $f$, defined as 
\[
\partial f=\left\{ y\in\left\{ -1,1\right\} ^{n}\Bigr|\exists i\text{ s.t }f\left(y\right)\neq f\left(y^{\oplus i}\right)\right\} .
\]
It is the disjoint union of the \emph{inner vertex boundary,} 
\[
\partial^{+}f=\left\{ y\in\left\{ -1,1\right\} ^{n}\Bigr|\exists i\text{ s.t }f\left(y\right)=1,f\left(y^{\oplus i}\right)=-1\right\} ,
\]
and the \emph{outer vertex boundary},
\[
\partial^{-}f=\left\{ y\in\left\{ -1,1\right\} ^{n}\Bigr|\exists i\text{ s.t }f\left(y\right)=-1,f\left(y^{\oplus i}\right)=1\right\} .
\]
The Cauchy-Schwartz inequality implies that $\mathcal{\mathbb{\mathbb{E}}}\sqrt{h_{f}}\leq\sqrt{\mathbb{E}\left[h_{f}\right]\mu\left(\partial f\right)}=\sqrt{\sum_{i}\Inf_{i}\left(f\right)\mu\left(\partial f\right)}$,
so the above conjecture strengthens the KKL result in the regime $\var\left(f\right)=\Omega\left(1\right)$
(see below Theorem \ref{thm:improved_talagrands_conjecture} for a
calculation). 

The inequality (\ref{eq:KKL}) was further generalized in another
direction by Talagrand \cite{talagrand_russo_approximate_zero_one_law},
who proved the following:
\begin{thm}
\label{thm:talagrands_inequality}There exists an absolute constant
$C_{T}>0$ such that for every $f:\left\{ -1,1\right\} ^{n}\to\left\{ -1,1\right\} $,
\begin{equation}
\var\left(f\right)\leq C_{T}\sum_{i=1}^{n}\frac{\Inf_{i}\left(f\right)}{1+\log\left(1/\Inf_{i}\left(f\right)\right)}.\label{eq:talangrands_main_theorem}
\end{equation}
\end{thm}

It is known that this inequality is sharp in the sense that for any
sequence of influences, there exist examples which saturate it \cite{klein_et_al_converse_of_talagrand_influence_inequality}.

Inequalities such as (\ref{eq:KKL}), (\ref{eq:talagrand_conjecture_first_appearance})
and (\ref{eq:talangrands_main_theorem}), in conjunction with concentration
of influence \cite{friedgut_junta} and sharp threshold properties
\cite{friedgut_sharp_thresholds_of_graph_properties}, have been widely
utilized across many subfields of mathematics and computer science,
including learning theory \cite{odnonnell_servedio_learning_monotone_decision_trees},
metric embeddings \cite{krauthgamer_rabani_improved_lower_bounds_for_embeddings_into_l1},
first passage percolation \cite{benjamini_kalai_schramm_first_passage_percolation},
classical and quantum communication complexity \cite{raz_communication_complexity,gavinsky_et_al_quantum_communication_complexity},
and hardness of approximation \cite{dinur_safra_hardness_of_approximating_minimum_vertex_cover};
and also in social network dynamics \cite{mossel_neeman_tamuz_majority_dynamics}
and statistical physics \cite{beffara_duminil_copin_self_dual_random_cluster_model}.
For a general survey, see \cite{kalai_safra_review}.

Talagrand's original proof of Theorem \ref{thm:talagrands_inequality},
as well as later proofs (see e.g \cite{cordero_ledoux_hypercontractive_measures}),
all rely on the hypercontractive principle. 

\subsection{Our results}

In this paper, we develop a new approach towards the proofs of the
aforementioned inequalities. Our proofs are based on pathwise analysis,
which bypasses the use of hypercontractivity, and in fact uses classical
Boolean Fourier-analysis only sparingly. Using these techniques, we
first to show that Talagrand's conjecture holds true:
\begin{thm}
\label{thm:talagrands_conjecture}There exists an absolute constant
$C>0$ such that for all $f:\left\{ -1,1\right\} ^{n}\to\left\{ -1,1\right\} $,
\[
\e\sqrt{h_{f}}\geq C\cdot\var\left(f\right)\cdot\sqrt{\log\left(2+\frac{e}{\sum_{i}\Inf_{i}\left(f\right)^{2}}\right)}.
\]
\end{thm}

In fact, we prove a stronger theorem, of which Theorem \ref{thm:talagrands_conjecture}
is an immediate corollary:
\begin{thm}
\label{thm:improved_talagrands_conjecture}There exists an absolute
constant $C>0$ such that the following holds. For all $f:\left\{ -1,1\right\} ^{n}\to\left\{ -1,1\right\} $,
there exists a function $g_{f}:\left\{ -1,1\right\} ^{n}\to\left[0,1\right]$
with $\e g^{2}\leq2\var\left(f\right)$ so that for all $1/2\leq p<1$,
\begin{equation}
\e\left[h_{f}^{p}g\right]\geq C\var\left(f\right)\cdot\left(\log\left(2+\frac{e}{\sum_{i}\Inf_{i}\left(f\right)^{2}}\right)\right)^{p}.\label{eq:improved_talagrands_conjecture}
\end{equation}
\end{thm}

The above inequality with $p=1/2$ implies both the KKL inequality
in full generality and a new lower bound on total influences in the
spirit of the isoperimetric inequality. By the Cauchy-Schwartz inequality,
\[
\e\left[\sqrt{h_{f}}g\right]\leq\sqrt{\e h_{f}}\sqrt{\e g^{2}}\leq\sqrt{\Inf\left(f\right)}\sqrt{2\var\left(f\right)},
\]
and plugging this into (\ref{eq:improved_talagrands_conjecture}),
we get 
\begin{equation}
\var\left(f\right)\leq C\cdot\frac{\sum_{i}\Inf_{i}\left(f\right)}{\log\left(\frac{e}{\sum_{i}\Inf_{i}\left(f\right)^{2}}\right)}.\label{eq:just_before_kkl}
\end{equation}
From this equation we can proceed in two directions:

\subparagraph*{Theorem \ref{thm:improved_talagrands_conjecture} $\protect\implies$
KKL}

Denoting $\delta=\max_{i}\Inf_{i}\left(f\right)$, the above display
yields 
\[
\var\left(f\right)\leq C\cdot\frac{\sum_{i}\Inf_{i}\left(f\right)}{\log\left(\frac{1}{\delta\sum_{i}\Inf_{i}\left(f\right)}\right)}.
\]

Consider now two cases: If $\delta\sum_{i}\Inf_{i}\left(f\right)\leq\delta^{1/2}$,
we immediately get from the above display that $\var\left(f\right)\leq2C\frac{\sum_{i}\Inf_{i}\left(f\right)}{\log\left(1/\delta\right)}$.
And if $\delta\sum_{i}\Inf_{i}\left(f\right)\geq\delta^{1/2}$, we
have 
\[
\sum_{i}\Inf_{i}\left(f\right)\geq\frac{1}{\delta^{1/2}}\geq\var\left(f\right)\frac{1}{\delta^{1/2}}\geq\frac{1}{2}\var\left(f\right)\log\left(\frac{1}{\delta}\right),
\]
(for $\delta<1$, otherwise there is nothing to prove), again yielding
$\var\left(f\right)\leq2\frac{\sum_{i}\Inf_{i}\left(f\right)}{\log\left(1/\delta\right)}$.

Hierarchically, the relation between the Poincaré inequality, KKL,
Talagrand's theorem and Theorem \ref{thm:improved_talagrands_conjecture}
may be summarized as in Figure \ref{fig:inequality-implications}.

\begin{figure}[H]
\begin{center}
\tikzstyle{block} = [rectangle, minimum width=3cm, minimum height=1cm,text centered, draw=black]

\tikzstyle{arrow} = [thick,->,>=stealth]

\begin{tikzpicture}[node distance=3cm] 
\node (poincare) [block] {Poincaré \eqref{eq:poincare_inequality}};

\node (dont_imply) [below of=poincare, align=center] {$\centernot{\Longrightarrow}$ \\ $\centernot{\Longleftarrow}$};

\node (kkl) [block, left of=dont_imply] {KKL \eqref{eq:KKL}};

\node (talagrand_theorem) [block, right of=dont_imply, xshift=0.5cm] {Talagrand's inequality \eqref{eq:talagrand_surface_area}};

\node (talagrand_conjecture) [block, below of=dont_imply] {Theorem \ref{thm:improved_talagrands_conjecture}};

\draw[implies-,double equal sign distance] (poincare) -- (kkl);

\draw[implies-,double equal sign distance] (poincare) -- (talagrand_theorem);

\draw[implies-,double equal sign distance] (talagrand_theorem) -- (talagrand_conjecture);

\draw[implies-,double equal sign distance] (kkl) -- node[anchor=east, xshift=-0.1cm, yshift=-0.1cm]{} (talagrand_conjecture);

\end{tikzpicture}
\end{center}

\caption{Inequality implications\label{fig:inequality-implications}}
\end{figure}
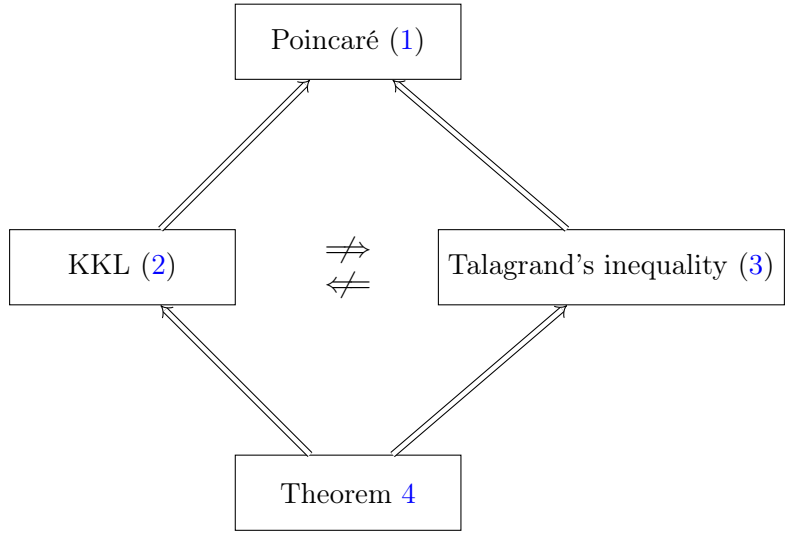

\subparagraph*{Theorem \ref{thm:improved_talagrands_conjecture} $\protect\implies$
Stability of the Isoperimetric inequality}

Assume that $\e f\leq0$ and let $A=\left\{ x\in\left\{ -1,1\right\} ^{n}\mid f\left(x\right)=1\right\} $
be the support of $f$, so that $\mu\left(A\right)\leq1/2$. The edge-isoperimetric
inequality \cite[section 3]{sergiu_a_note_on_the_edges_of_the_n_cube}
states that 
\begin{equation}
\sum_{i=1}^{n}\Inf_{i}\left(f\right)\geq2\mu\left(A\right)\log_{2}\frac{1}{\mu\left(A\right)},\label{eq:isoperimetric_inequality}
\end{equation}
with equality if and only if $A$ is a subcube. Suppose that $f$
saturates the isoperimetric inequality up to a constant, i.e 
\[
\sum_{i=1}^{n}\Inf_{i}\left(f\right)\leq C\mu\left(A\right)\log_{2}\frac{1}{\mu\left(A\right)}
\]
for some constant $C$. Since $\var\left(f\right)=4\mu\left(A\right)\left(1-\mu\left(A\right)\right)\geq2\mu\left(A\right)$,
this gives 
\begin{equation}
\sum_{i=1}^{n}\Inf_{i}\left(f\right)\leq C\var\left(f\right)\log_{2}\frac{2}{\var\left(f\right)}.\label{eq:sort_of_iso_but_not}
\end{equation}
Suppose also that $f$ is monotone; then $\var\left(f\right)\geq\sum_{i}\Inf_{i}\left(f\right)^{2}$,
and from equations (\ref{eq:just_before_kkl}) and (\ref{eq:sort_of_iso_but_not})
we get that 
\[
\var\left(f\right)\leq C\frac{\sum_{i}\Inf_{i}\left(f\right)}{\log\left(\frac{2}{\sum_{i}\Inf_{i}\left(f\right)^{2}}\right)}\leq C\frac{\sum_{i}\Inf_{i}\left(f\right)}{\log\left(\frac{2}{\var\left(f\right)}\right)}\leq C'\var\left(f\right).
\]
In particular, the two denominators are within a constant factor of
each other:
\[
\log\left(\frac{2}{\sum_{i}\Inf_{i}\left(f\right)^{2}}\right)=\Theta\left(\log\left(\frac{2}{\var\left(f\right)}\right)\right),
\]
implying that the Fourier mass on the first level is proportional
to a power of the variance.

Next, we reprove Theorem \ref{thm:talagrands_inequality} using stochastic
techniques, and provide a strengthening which can be thought of as
a stability version of this bound in terms of the vertex boundary
of $f$: If near-equality is attained in equation (\ref{eq:talangrands_main_theorem}),
then both the inner and outer vertex boundaries of $f$ are large.
The theorem reads,
\begin{thm}
\label{thm:robust_implies_large_vertex_boundary}Let $T\left(f\right)=\sum_{i=1}^{n}\frac{\Inf_{i}\left(f\right)}{1+\log\left(1/\Inf_{i}\left(f\right)\right)}$,
and denote $r_{\mathrm{Tal}}=\frac{\var\left(f\right)}{T\left(f\right)}$.
There exists an absolute constant $C_{B}>0$ such that
\begin{align*}
\mu\left(\partial^{\pm}f\right) & \geq\frac{r_{\mathrm{Tal}}}{C_{B}\log\frac{C_{B}}{r_{\mathrm{Tal}}}}\var\left(f\right).
\end{align*}
\end{thm}

Theorem \ref{thm:robust_implies_large_vertex_boundary} can be readily
applied to two related functional inequalities - the isoperimetric
inequality and the KKL inequality - showing that when either of the
inequalities are tight up to a constant, the function must have a
large vertex boundary.

\subparagraph*{The Isoperimetric inequality and vertex boundary}

It is natural to ask about the robustness of the isoperimetric inequality:
Is it true that if near-equality is attained in (\ref{eq:isoperimetric_inequality}),
then $A$ is close to a subcube in some sense? This question was answered
in \cite{ellis_almost_isoperimetric_subsets} for sets $A$ which
are $\left(1+\eps\right)$-close to satisfying the inequality. Conjectures
concerning sets for which the inequality is tight only up to a constant
multiplicative factor can be found in \cite{kahn_kalai_thresholds_and_expectation_thresholds}.
We make a step in this direction by giving the first bound which is
meaningful when the function is $O\left(1\right)$-close to satisfying
the inequality (\ref{eq:isoperimetric_inequality}), showing that
in that case, a constant proportion of the set $A$ is in its inner
vertex boundary (whereas for the extremizers, the vertex boundary
is the entire set $A$). 
\begin{cor}
Let $r_{\mathrm{Iso}}=\frac{2\mu\left(A\right)\log_{2}\frac{1}{\mu\left(A\right)}}{\sum_{i=1}^{n}\Inf_{i}\left(f\right)}.$
Then there exists a constant $c_{\mathrm{Iso}}\geq\frac{r_{\mathrm{Iso}}}{2C_{B}\log\left(\frac{2C_{B}}{r_{\mathrm{Iso}}}\right)}$
depending only on $r_{\mathrm{Iso}}$ such that 
\[
\mu\left(\partial A\right)\geq c_{\mathrm{Iso}}\mu\left(A\right).
\]
\end{cor}

\begin{proof}
As in Theorem \ref{thm:robust_implies_large_vertex_boundary}, denote
$r_{\mathrm{Tal}}=\frac{\var\left(f\right)}{T\left(f\right)}$. Observe
that for every index $i$, $\Inf_{i}\left(f\right)\leq2\mu\left(A\right)$.
Since $\mu\left(A\right)\leq1/2$, we have 

\begin{align*}
\var\left(f\right) & =4\mu\left(A\right)\left(1-\mu\left(A\right)\right)\geq2\mu\left(A\right).
\end{align*}
This gives a bound on $r_{\mathrm{Tal}}$:
\begin{align*}
r_{\mathrm{Tal}} & =\frac{\var\left(f\right)}{\sum_{i}^{n}\frac{\Inf_{i}\left(f\right)}{1+\log\left(1/\Inf_{i}\left(f\right)\right)}}\geq\frac{\var\left(f\right)}{\sum_{i}^{n}\frac{\Inf_{i}\left(f\right)}{1+\log\left(1/2\mu\left(A\right)\right)}}\\
 & =\frac{r_{\mathrm{Iso}}\var\left(f\right)\left(1+\log\left(\frac{1}{2\mu\left(A\right)}\right)\right)}{\frac{2}{\log2}\mu\left(A\right)\log\frac{1}{\mu\left(A\right)}}\geq\frac{r_{\mathrm{Iso}}\log2\cdot\var\left(f\right)}{2\mu\left(A\right)}\geq\frac{r_{\mathrm{Iso}}}{2}.
\end{align*}
Thus, by Theorem \ref{thm:robust_implies_large_vertex_boundary},
there exists a constant $c_{\mathrm{Iso}}\geq\frac{r_{\mathrm{Iso}}}{2C_{B}\log\left(\frac{2C_{B}}{r_{\mathrm{Iso}}}\right)}$
such that 
\[
\mu\left(\partial^{\pm}f\right)\geq\frac{c_{\mathrm{Iso}}}{2}\var\left(f\right)\ge c_{\mathrm{Iso}}\mu\left(A\right).
\]
\end{proof}

\subparagraph*{The KKL inequality and vertex boundary}

In its original formulation, the KKL theorem \cite[Theorem 3.1]{KKL},
which follows immediately from (\ref{eq:KKL}), states that a Boolean
function must have a variable with a relatively large influence: There
exists an absolute constant $C>0$ such that for every $f:\left\{ -1,1\right\} ^{n}\to\left\{ -1,1\right\} $,
there exists an index $i\in\left[n\right]$ with
\[
\Inf_{i}\left(f\right)\geq C\cdot\var\left(f\right)\frac{\log n}{n}.
\]
Our second corollary states that if all influences are of the order
$\var\left(f\right)\frac{\log n}{n}$ , then the function must have
a large (inner and outer) vertex boundary.
\begin{cor}
Suppose that for some $C\leq\sqrt{n}$, we have $\Inf_{i}\left(f\right)\leq C\cdot\var\left(f\right)\frac{\log n}{n}$
for all $i$. Then there exists a constant $c_{\mathrm{KKL}}$ depending
only on $C$ such that 
\[
\mu\left(\partial^{\pm}f\right)\geq c_{\mathrm{KKL}}\var\left(f\right).
\]
\end{cor}

\begin{proof}
In this case, we have 
\begin{align*}
r_{\mathrm{Tal}} & =\frac{\var\left(f\right)}{\sum_{i}^{n}\frac{\Inf_{i}\left(f\right)}{1+\log\left(1/\Inf_{i}\left(f\right)\right)}}\geq\frac{\var\left(f\right)}{\sum_{i}^{n}\frac{\Inf_{i}\left(f\right)}{1+\log\left(C\cdot\var\left(f\right)\frac{\log n}{n}\right)}}\\
 & \geq\frac{\var\left(f\right)\left(1+\log\left(\frac{n}{C\var\left(f\right)\log n}\right)\right)}{C\cdot\var\left(f\right)\log n}\geq\frac{\log n-\log\left(C\var\left(f\right)\log n\right)}{C\log n}>\frac{1}{4C}.
\end{align*}
Thus, by Theorem \ref{thm:robust_implies_large_vertex_boundary},
there exists a constant $c_{\mathrm{KKL}}$ which depends only on
$C$ such that 
\[
\mu\left(\partial^{\pm}f\right)\geq c_{\mathrm{KKL}}\var\left(f\right).
\]
\end{proof}
Finally, we improve an inequality by Keller and Kindler \cite{keller_kindler_quantitative_noise_sensitivity}.
Let $S_{\eps}\left(f\right)$ be the noise stability of $f$, i.e
\[
S_{\eps}\left(f\right)=\mathrm{Cov}_{x\sim\mu,y\sim N_{\eps}\left(x\right)}\left[f\left(x\right),f\left(y\right)\right],
\]
where $N_{\eps}\left(x\right)$ is a random vector whose $i$-th coordinate
is equal to $x_{i}$ with probability $1-\eps$ and to a uniformly
random bit with probability $\eps$. 
\begin{thm}
\label{thm:improved_keller_kindler}There exists universal constants
$C,c>0$ such that 
\begin{equation}
S_{\eps}\left(f\right)\leq C\cdot\var\left(f\right)\left(\sum_{i=1}^{n}\Inf_{i}\left(f\right)^{2}\right)^{c\eps}.\label{eq:improved_keller_kindler}
\end{equation}
\end{thm}

The bound proved in \cite{keller_kindler_quantitative_noise_sensitivity}
is the same but with the term $\var\left(f\right)$ is replaced by
a constant; thus our result becomes stronger when $\var\left(f\right)=o\left(1\right).$
This theorem is used in the proof of Theorem \ref{thm:talagrands_inequality}.
The relation between influences and noise sensitivity was first established
in \cite{benjamini_kalai_schramm_noise_sensitivity}, where a qualitative
bound of the same nature is proven.

\subsection{Proof outline}

The core of our proofs is the construction of a martingale $B_{t}=\left(B_{t}^{\left(1\right)},\ldots,B_{t}^{\left(n\right)}\right)\in\r^{n}$
which satisfies $\abs{B_{t}^{\left(i\right)}}=t$ and $B_{1}\sim\mathrm{Unif}\left(\left\{ -1,1\right\} ^{n}\right)$
(Proposition \ref{prop:existence_of_martingale}).

Since $B_{1}$ is uniform on the hypercube, the expected value and
variance of $f$ can be obtained by $\e f=\e f\left(B_{1}\right)$
and $\var f=\var f\left(B_{1}\right)$, where the expectations in
the right hand sides are over the randomness of the process $B_{t}$.
Similarly, the influence of the $i$-th bit is given by $\Inf_{i}\left(f\right)=\e\partial_{i}f\left(B_{1}\right)^{2}$,
where $\partial_{i}f$ is the partial derivative of $f$ in direction
$i$, and $\e h_{f}^{p}$ is given by $\e\norm{\grad f\left(B_{1}\right)}_{2}^{2p}$. 

The strength of the stochastic process approach stems from the fact
that the behavior of $f\left(B_{1}\right)$ can be understood by an
analysis of the processes $B_{t}$, $f\left(B_{t}\right)$ and $\grad f\left(B_{t}\right)$
for times smaller than $1$. Indeed, there is a natural way to extend
the domain of a Boolean function to the continuous hypercube $\left[-1,1\right]^{n}$
so that the processes $f\left(B_{t}\right)$ and $\partial_{i}f\left(B_{t}\right)$
become martingales. The variance of $f$ can then be expressed as
\[
\var\left(f\right)=2\e\sum_{i=1}^{n}\int_{0}^{1}t\left(\partial_{i}f\left(B_{t}\right)\right)^{2}dt
\]
(Lemma \ref{lem:sum_of_jumps_to_integral} and Corollary \ref{cor:variance_from_sum_of_jumps}).
Bounding the variance is then a matter of bounding the integral $\e\sum_{i}\int_{0}^{1}t\left(\partial_{i}f\left(B_{t}\right)\right)^{2}dt$,
and for this we can utilize tools from real analysis and stochastic
processes. Specifically, two well-known inequalities - called the
Level-1 and Level-2 inequalities - give us bounds on the speed with
which both the individual processes $\partial_{i}f\left(B_{t}\right)^{2}$
and their collective sum $\sum\left(\partial_{i}f\left(B_{t}\right)\right)_{i}^{2}$
are moving in terms of their current value. In the Gaussian setting,
somewhat similar ideas of using level inequalities appear in \cite{eldan_two_sided_estimate_for_gaussian_noise_stabiltiy}.

This points to a significant conceptual difference between existing
techniques that use the hypercontractivity of the heat operator and
our technique: Whereas the former proofs start from the function $f$
and analyze the way that it changes by applying the heat semigroup,
which corresponds to going backwards in the time $t$, our analysis
goes forward in time. We may think of the process $B_{t}$ as a way
to sample from $\left\{ -1,1\right\} ^{n}$ via a continuous filtration,
where we add ``infinitesimal bits of randomness'' as time progresses.
The analysis starts from $f\left(B_{0}\right)=\mathbb{E}f$ and considers
the way that the martingales evolve as we refine our filtration, or
in other words, add more randomness. On a first glance this difference
may seem to be only pedagogical, but the strength of our approach
is that the pathwise analysis equips us with new tools, such as using
stopping times and the optional stopping theorem, and conditioning
on the past.

Towards proving Talagrand's conjecture, we use the Level-2 inequality
(Lemma \ref{lem:biased_level_2_inequality}) to bound $\sum_{i}\left(\partial_{i}f\left(B_{t}\right)\right)^{2}$
by a time-dependent power of the sum of squares of influences $\sum_{i}\Inf_{i}\left(f\right)^{2}$.
Ideologically, when this sum is small, this roughly implies that the
process$\norm{\grad f\left(B_{t}\right)}_{2}^{2}$ makes most of its
movement very close to time $t=1$; this is in fact the essence of
Theorem \ref{thm:improved_keller_kindler}. This can then be used
to show that most of the quadratic variation of $f\left(B_{t}\right)$
comes from paths in which there is a time $t$ such that $\norm{\grad f\left(B_{t}\right)}_{2}$
is larger than $\alpha\left(\log\left(\frac{1}{\sum\Inf_{i}^{2}\left(f\right)}\right)\right)^{1/2}$
(Proposition \ref{prop:left_quadratic_variation_of_small_gradient_is_small}).
However, the quadratic variation is itself large with probability
that is directly proportional to the variance (Proposition \ref{prop:quadratic_variation_is_not_concentrated_near_zero}).
This is one of the steps where the pathwise analysis is crucially
used; without it, we would have only known that the quadratic variation
is large in expectation, which would not eliminate the possibility
that the entire contribution to the variance is made on an event of
negligible probability. Since $\norm{\grad f\left(B_{1}\right)}_{2}^{2p}$
is a submartingale, if there was ever a time when $\norm{\grad f\left(B_{1}\right)}_{2}^{2p}$
is large, then in expectation it continues to be large. Thus $\e\norm{\grad f\left(B_{1}\right)}_{2}^{2p}$
is larger than $\alpha\var\left(f\right)\left(\log\left(\frac{1}{\sum\Inf_{i}^{2}\left(f\right)}\right)\right)^{p}$,
giving the original Talagrand's conjecture (Theorem \ref{thm:talagrands_conjecture})
when $p=1/2$. With additional care, it can be shown that $f\left(B_{t}\right)$
itself is large at some time before the gradient was large, which
gives the strengthened result (Theorem \ref{thm:improved_talagrands_conjecture}). 

This is a good place to point out an analogy between our technique
and the one demonstrated by Barthe and Maurey \cite{barthe_maurey_some_remarks_on_isoperimetry_of_gaussian_type},
who give a stochastic proof of Bobkov's extension of inequality (\ref{eq:talagrand_surface_area}).
They use a stochastic argument in order to derive a one-dimensional
inequality which implies Bobkov's inequality via tensorization, and
one of the central components in their proof is to establish that
a certain process which is analogous to $\norm{\grad f\left(B_{1}\right)}_{2}$
is a submartingale (this is based on ideas introduced in a paper by
Capitaine, Hsu and Ledoux \cite{capitaine_hsu__ledoux_martingale_representation}).

Similarly, for proving Theorem \ref{thm:talagrands_inequality}, we
use the Level-1 inequality (Lemma \ref{lem:biased_level_1_inequality})
to bound each individual $\left(\partial_{i}f\left(B_{t}\right)\right)^{2}$
by a time-dependent power of the influence $\Inf_{i}\left(f\right)$
(Lemma \ref{lem:phi_can_be_bounded}). When the influences are small,
this roughly implies that the martingale $f\left(B_{t}\right)$ makes
most of its movement very close to time $t=1$. Theorem \ref{thm:talagrands_inequality}
then follows by plugging this bound into the integral (Proposition
\ref{prop:Q_integral_is_smaller_than_phi_etc}).

The proof of Theorem \ref{thm:robust_implies_large_vertex_boundary}
is more involved, and utilizes the fact that $f\left(B_{t}\right)$
is both a jump process and a martingale: For such processes, the variance
of $f\left(B_{1}\right)$ is then given by the sum of squares of jumps
of $f\left(B_{t}\right)$ up to time $1$:
\[
\var\left(f\right)=\e\sum_{s\in\mathrm{Jump}\left(B_{t}\right)}\left(\Delta f\left(B_{s}\right)\right)^{2}=2\e\sum_{i=1}^{n}\int_{0}^{1}t\left(\partial_{i}f\left(B_{t}\right)\right)^{2}dt.
\]
The technical core of the proof (Proposition \ref{prop:integral_cannot_be_too_small}
and Lemma \ref{lem:bounding_the_probability_of_F}) shows that if
$T\left(f\right)$ and $\var\left(f\right)$ differ only by a multiplicative
constant, then with non-negligible probability the process $f\left(B_{t}\right)$
must make a relatively large jump somewhere along the way. Roughly
speaking, this is because if the process $B_{t}^{\left(i\right)}$
jumps at time $t$, then the function $f\left(B_{t}\right)$ also
jumps, changing by a value of $2t\partial_{i}f\left(B_{t}\right)$.
If all the jumps are small, then the expression $\sum_{s\in\mathrm{Jump}\left(B_{t}\right)}\left(\Delta f\left(B_{s}\right)\right)^{2}$
in the left hand side of the above display (which cares only about
jumps) must be substantially smaller than the integral in the right
hand side (which cares only about the size of the derivatives). 

Now, when the process $f\left(B_{t}\right)$ makes a large jump, it
necessarily means that the magnitude of one of the partial derivatives
$\partial_{i}f$ is large. Since the process $\partial_{i}f\left(B_{t}\right)$
is also a martingale, if it is large at some point in time, then it
continues to be large with relatively high probability. But at time
$t=1$, since $B_{1}$ is uniform on the hypercube, the only possibilities
for the values of $\partial_{i}f\left(B_{1}\right)$ are $-1$, $0$
and $1$. Thus, it is likely that $\abs{\partial_{i}f\left(B_{1}\right)}=1$.
This exactly corresponds to the point $B_{1}$ being in the vertex
boundary, showing that the vertex boundary is large. The distinction
between the inner and outer vertex follows by similar arguments, using
a symmetrification of $B_{t}$ (Proposition \ref{prop:using_the_event_E}).

\subsection*{Acknowledgements}

R.E. would like to thank Noam Lifshitz for useful discussions and
in particular for pointing out the possible application to stability
of the isoperimetric inequality. We are also thankful to Ramon Van
Handel, Itai Benjamini and Gil Kalai for an enlightening discussion,
and grateful for the comments from the anonymous STOC referees.

\section{Preliminaries}

Throughout the text, the letter $C$ stands for a positive universal
constant, whose value may change from line to line. 

\subsection{Boolean functions}

For a general introduction to Boolean functions, see \cite{odonnell_analysis_of_boolean_functions};
in what follows, we provide a brief overview of the required background
and notation.

Every Boolean function $f:\left\{ -1,1\right\} ^{n}\to\r$ may be
uniquely written as a sum of monomials:
\begin{equation}
f\left(y\right)=\sum_{S\subseteq\left[n\right]}\hat{f}\left(S\right)\prod_{i\in S}y_{i},\label{eq:definition_of_fourier_decomposition}
\end{equation}
where $\left[n\right]=\left\{ 1,\ldots,n\right\} $, and the harmonic
coefficients (also known as Fourier coefficients) $\hat{f}\left(S\right)$
are given by 
\begin{equation}
\hat{f}\left(S\right)=\e\left[f\left(y\right)\prod_{i\in S}y_{i}\right].\label{eq:definition_of_fourier_coefficient}
\end{equation}
Equation (\ref{eq:definition_of_fourier_decomposition}) may be used
to extend a function's domain from the discrete hypercube $\left\{ -1,1\right\} ^{n}$
to real space $\r^{n}$. We call this the \emph{harmonic extension},
and denote it also by $f$. Under this notation, $f\left(0\right)=\e f$.
In general, for $x\in\left[-1,1\right]^{n}$, the harmonic extension
$f\left(x\right)$ is a convex combination of $f$'s values on all
the points $y\in\left\{ -1,1\right\} ^{n}$:
\begin{equation}
f\left(x\right)=\sum_{y\in\left\{ -1,1\right\} ^{n}}w_{x}\left(y\right)f\left(y\right),\label{eq:harmonic_extension_is_weighted_sum}
\end{equation}
where $w_{x}\left(y\right)=\prod_{i=1}^{n}\left(1+x_{i}y_{i}\right)/2$. 

The derivative of a function $f$ in direction $i$ is defined as
\[
\partial_{i}f\left(y\right)=\frac{f\left(y^{i\to1}\right)-f\left(y^{i\to-1}\right)}{2},
\]
where $y^{i\to a}$ has $a$ at coordinate $i$, and is identical
to $y$ at all other coordinates. The gradient is then defined as
$\grad f=\left(\partial_{1}f,\ldots,\partial_{n}f\right)$. A function
is called \emph{monotone} if $f\left(x\right)\leq f\left(y\right)$
whenever $x_{i}\leq y_{i}$ for all $i\in\left[n\right]$. Similar
to the function $f$, by abuse of notation $\partial_{i}f$ will denote
the harmonic extension of $\partial_{i}f$, and we will treat it as
a function on $\left[-1,1\right]^{n}$.

A short calculation reveals the following properties of the derivative:
\begin{enumerate}
\item The harmonic extension of the derivative $\partial_{i}f$ is equal
to the real-differentiable partial derivative $\frac{\partial}{\partial x_{i}}$
of the harmonic extension of $f$.
\item For functions whose range is $\left\{ -1,1\right\} $, the derivative
$\partial_{i}f$ takes values in $\left\{ -1,0,1\right\} $, and the
influence of the $i$-th coordinate of $f$ is given by 
\begin{align}
\Inf_{i}\left(f\right) & =\e\left(\partial_{i}f\left(y\right)\right)^{2}=\e\abs{\partial_{i}f\left(y\right)}.\label{eq:influence_by_derivative}
\end{align}
\item For monotone functions, the derivative $\partial_{i}f$ only takes
values in $\left\{ 0,1\right\} $, and the influence of the $i$-th
coordinate is then given by 
\begin{equation}
\Inf_{i}\left(f\right)=\e\partial_{i}f\left(y\right)=\hat{f}\left(\left\{ i\right\} \right).\label{eq:influence_of_monotone_function}
\end{equation}
\end{enumerate}
In the definition of the Fourier coefficient in (\ref{eq:definition_of_fourier_coefficient}),
the expectation is over the uniform measure $\mu\left(y\right)=\frac{1}{2^{n}}$.
It is also possible to decompose a function into Fourier coefficients
over a biased measure. This type of analysis will be used only in
the proof of Theorem \ref{thm:improved_keller_kindler}. A brief overview
can be found in the appendix.

Finally, we'll require two lemmas which effectively relate the weights
of the Fourier coefficients at higher levels with those of lower ones;
this translates to inequalities between the harmonic extension of
a function and its derivatives. The first lemma is a direct application
of the fact that $w_{x}\left(\cdot\right)$ is subgaussian; it essentially
bounds the Fourier weights in the first level by a function of the
weights at level zero:
\begin{lem}[Level-1 inequality]
\label{lem:biased_level_1_inequality}There exists a constant $L$
so that the following holds. Let $g:\left[-1,1\right]^{n}\to\left[0,1\right]$
be the harmonic extension of a Boolean function, and let $x\in\left(-1,1\right)^{n}$
be such that $\abs{x_{i}}=t$ for all $i$. Then 
\begin{equation}
\norm{\grad g\left(x\right)}_{2}^{2}\leq\frac{L}{\left(1-t\right)^{4}}g\left(x\right)^{2}\log\frac{e}{g\left(x\right)}.\label{eq:bound_on_gradient}
\end{equation}
\end{lem}

The second lemma, whose original, uniform case is due to Talagrand
\cite{talagrand_how_much_are_increasing_sets_positively_correlated},
essentially bounds the Fourier weights in the second level by those
of the first. It is similar to \cite[Lemma 6]{keller_kindler_quantitative_noise_sensitivity},
but for real-valued functions.
\begin{lem}[Level-2 inequality]
\label{lem:biased_level_2_inequality}There exists a continuous function
$C:\left[0,1\right)\to\left[0,\infty\right)$ so that the following
holds. Let $g:\left[-1,1\right]^{n}\to\left[-1,1\right]$ be the harmonic
extension of a monotone function, and let $x\in\left(-1,1\right)^{n}$
be such that $\abs{x_{i}}=t$ for all $i$. Then 
\begin{equation}
\norm{\nabla^{2}g\left(x\right)}_{HS}^{2}\leq C\left(t\right)\norm{\grad g\left(x\right)}_{2}^{2}\cdot\log\left(\frac{C\left(t\right)}{\norm{\grad g\left(x\right)}_{2}^{2}}\right),\label{eq:quantitative_noise_sensitivity_simplified-1}
\end{equation}
where $\nabla^{2}g$ is the Hessian $\left(\partial_{i}\partial_{j}g\right)_{i,j=1}^{n}$
of $g$, and $\norm X_{HS}=\sqrt{\sum_{i,j}X_{ij}^{2}}$ is the Hilbert-Schmidt
norm of a matrix.
\end{lem}

\begin{rem*}
In both lemmas, the requirement that $\abs{x_{i}}=t$ for all $i$
is not crucial, and can be replaced by $\abs{x_{i}}\leq t$ for all
$i$.
\end{rem*}
The proofs of both lemmas are found in the appendix. 

\subsection{Stochastic processes and quadratic variation}

For a general introduction to stochastic processes and Poisson processes,
see \cite{durrett_probability_theory_and_examples} and \cite{kingman_poisson_processes}.

A Poisson point process $N_{t}$ with rate $\lambda\left(t\right)$
is an integer-valued process such that $N_{0}=0$, and for every $0\leq a<b$,
the difference $N_{b}-N_{a}$ distributes as a Poisson random variable
with rate $\int_{a}^{b}\lambda\left(t\right)dt$. If $\int_{a}^{b}\lambda\left(t\right)dt<\infty$
for all $0\leq a<b$, then the sample-paths of a Poisson point process
are right-continuous almost surely. The (random) set of times at which
the sample-path is discontinuous is denoted by $\mathrm{Jump}\left(N_{t}\right)$. 

Let $\lambda\left(t\right)$ be such that $\int_{a}^{b}\lambda\left(t\right)dt<\infty$
for all $0\leq a<b$ and let $N_{t}$ be a Poisson point process with
rate $\lambda\left(t\right)$. The set $\mathrm{Jump}\left(N_{t}\right)=\left\{ t_{1},t_{2},\ldots\right\} $
is then almost surely discrete. A process $X_{t}$ is said to be a
\emph{piecewise-smooth jump process with rate }$\lambda\left(t\right)$
if $X_{t}$ is right-continuous and is smooth in the interval $\left[t_{i},t_{i+1}\right)$
for every $i=1,2,\ldots$. This definition can be extended to the
case where $\int_{0}^{b}\lambda\left(t\right)dt=\infty$ but $\int_{a}^{b}\lambda\left(t\right)dt<\infty$
for all $0<a<b$ (this happens, for example, when $\lambda=1/t$):
In this case $\mathrm{Jump}\left(N_{t}\right)$ has only a single
accumulation point at $0$, and intervals between successive jump
times are still well defined. 

An important notion in the analysis of stochastic processes is \emph{quadratic
variation}. Intuitively, the quadratic variation of a process $X_{t}$,
denoted $\left[X\right]_{t}$, describes how wildly the process $X_{t}$
fluctuates; formally, it is defined as 
\[
\left[X\right]_{t}=\lim_{\norm P\to0}\sum_{k=1}^{n}\left(X_{t_{k}}-X_{t_{k-1}}\right)^{2},
\]
if the limit exists; here $P$ is an $n$-part partition of $\left[0,t\right]$,
and the notation $\lim_{\norm P\to0}$ indicates that the size of
the largest part goes to $0$. Not all processes have a (finite) quadratic
variation, but piecewise-smooth jump processes do; in fact, it can
be seen from definition that if $X_{t}$ is a piecewise-smooth jump
process then 
\begin{equation}
\left[X\right]_{t}=\sum_{s\in\mathrm{Jump\left(X_{t}\right)}\intersect\left[0,t\right]}\left(\Delta X_{s}\right)^{2},\label{eq:quadratic_variation_of_jump_process}
\end{equation}
where $\Delta X_{s}=\lim_{\eps\to0^{+}}\left(X_{s+\eps}-X_{s-\eps}\right)$
is the size of the jump at time $s$.

The quadratic variation is especially useful for martingales due to
its relation with the variance: If $X_{t}$ is a martingale, then
\begin{equation}
\var\left(X_{t}\right)=\e\left(\left[X\right]_{t}\right).\label{eq:variance_of_martingales_via_quadratic_variation}
\end{equation}

\section{The main tool: A jump process}

The proof of Theorems \ref{thm:talagrands_inequality} and \ref{thm:robust_implies_large_vertex_boundary}
relies on the construction of a piecewise-smooth jump process martingale
$B_{t}$, described below. One of its key properties is that it will
allow us to express quantities such as the variance of $f$ in terms
of derivatives of the harmonic extension, e.g:

\[
\var\left(f\right)=2\e\sum_{i=1}^{n}\int_{0}^{1}t\left(\partial_{i}f\left(B_{t}\right)\right)^{2}dt.
\]
The process $\left(B_{t}\right)_{t\geq0}$ is characterized by the
following properties:
\begin{enumerate}
\item $B_{t}\in\r^{n}$, with $B_{t}^{\left(i\right)}$ independent and
identically distributed for all $i\in\left[n\right]$.
\item $B_{t}^{\left(i\right)}$ is a martingale for all $i$.
\item $\abs{B_{t}^{\left(i\right)}}=t$ almost surely for all $i\in\left[n\right]$
and $t\geq0$.
\end{enumerate}
\begin{prop}
\label{prop:existence_of_martingale}There exists a right continuous
martingale with the above properties. Furthermore, for all $t,h>0$,
\begin{equation}
\p\left[\sign B_{t+h}^{\left(i\right)}\neq\sign B_{t}^{\left(i\right)}\mid B_{t}\right]=\frac{h}{2\left(t+h\right)}.\label{eq:sign_jumps_probabilities}
\end{equation}
\end{prop}

\begin{proof}
Let $W_{s}$ be a standard Brownian motion. Consider the family of
stopping times
\[
\tau\left(t\right)=\inf\left\{ s>0\Bigr|\abs{W_{s}}>t\right\} 
\]
and define $X_{t}=W_{\tau\left(t\right)}$. Then by definition, $\abs{X_{t}}=t$,
and $X_{t}$ is a martingale due to the optional stopping theorem.
Observe that $X_{t}$ can fail to be right-continuous only if $\sign W_{\tau\left(t\right)}$
is different from $\sign W_{\tau\left(s\right)}$ for all $s\neq t$
in some open interval around $t$. This event happens with probability
$0$, and so there exists a modification of $X_{t}$ where paths are
right-continuous almost surely. The process $B_{t}$ is defined as
$B_{t}=\left(X_{t}^{\left(1\right)},\ldots,X_{t}^{\left(n\right)}\right)$,
where $X_{t}^{\left(i\right)}$ are independent copies of $X_{t}$.

To prove equation (\ref{eq:sign_jumps_probabilities}), set $p=\p\left(\sign X_{t+h}\neq\sign X_{t}\mid X_{t}\right)$
and use the martingale property:
\begin{align*}
t\sign X_{t} & =X_{t}\\
 & =\e\left[X_{t+h}\Bigr|X_{t}\right]\\
 & =\left(-t-h\right)\sign X_{t}\cdot p+\left(1-p\right)\left(t+h\right)\sign X_{t}.
\end{align*}
Rearranging gives $p=\frac{h}{2\left(t+h\right)}$ as needed.
\end{proof}
It can be readily seen that $B_{t}^{\left(i\right)}$ is a piecewise-smooth
jump process with rate $\lambda\left(t\right)=1/2t$. Denote its set
of discontinuities by $J_{i}=\mathrm{Jump}\left(B_{t}^{\left(i\right)}\right)$.

As described in (\ref{eq:definition_of_fourier_decomposition}), the
harmonic extension of a function $f:\left\{ -1,1\right\} ^{n}\to\r$
is a multilinear polynomial. Since the product of two independent
martingales is also a martingale with respect to its natural filtration,
by independence of the coordinates of $B_{t}$, we conclude that
\begin{fact}
For a function $f:\left\{ -1,1\right\} ^{n}\to\r$, the process $f\left(B_{t}\right)$
is a martingale.
\end{fact}

We denote this process by $f_{t}=f\left(B_{t}\right)$, and by slight
abuse of notation, write $\partial_{i}f_{t}=\partial_{i}f\left(B_{t}\right)$
and $\grad f_{t}=\grad f\left(B_{t}\right)$. Since $B_{t}$ is right-continuous,
these processes are right-continuous also; when referring to the left
limit at jump discontinuities, we write $f_{t^{-}}$, $\partial_{i}f_{t^{-}}$
and $\grad f_{t^{-}}$, with $f_{t^{-}}=\lim_{\eps\searrow0}f_{t-\eps}$.
Some example sample paths of $f_{t}$ for the $15$-bit majority function
are given in Figure \ref{fig:examples_of_f_b_t}.

\begin{figure}
\begin{centering}
\includegraphics[scale=0.5]{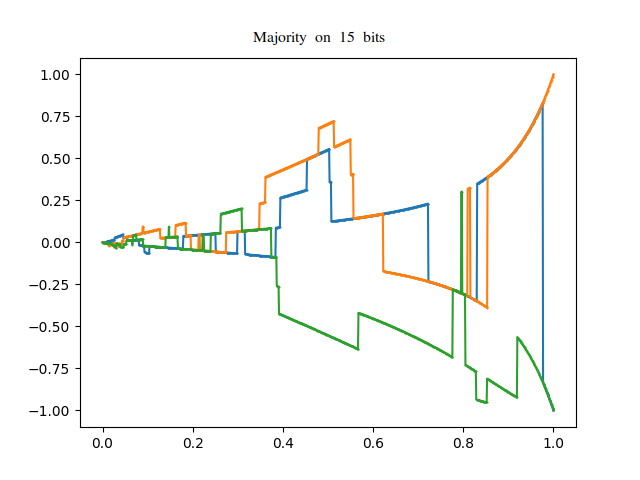}
\par\end{centering}
\caption{Sample paths of $f_{t}$ for the $15$-bit majority function\label{fig:examples_of_f_b_t}}
\end{figure}

Since $f_{t}$ is a piecewise-smooth jump process, by (\ref{eq:quadratic_variation_of_jump_process})
its quadratic variation is equal to the sum of squares of its jumps.
Now, almost surely, $B_{t}$ can make a jump only in one coordinate
at a time, and when the $i$-th coordinate jumps, the value of $f_{t}$
changes by $2t\partial_{i}f_{t}$, since $f$ is multi-linear. The
quadratic variation of $f_{t}$ is therefore
\begin{equation}
\left[f\right]_{t}=\sum_{i=1}^{n}\sum_{s\in J_{i}\intersect\left[0,t\right]}\left(2s\cdot\partial_{i}f_{s}\right)^{2}.\label{eq:quadratic_variation_of_m}
\end{equation}
A crucial property of $B_{t}$ is that the expected value of these
jumps behaves smoothly, as the next lemma shows:
\begin{lem}
\label{lem:sum_of_jumps_to_integral} Let $0\leq t_{1}<t_{2}\leq1$,
and let $g_{t}$ be a bounded process which satisfies one of the following: 
\begin{enumerate}
\item \label{enu:sum_of_jumps_left}$g_{t}$ is left-continuous and measurable
with respect to the filtration generated by $\left\{ B_{s}\right\} _{0\leq s<t}$. 
\item \label{enu:sum_of_jumps_right}There exists a continuous function
$g:\left[-1,1\right]^{n}\to\r$ such that $g_{t}=g\left(B_{t}\right)$.
\end{enumerate}
Then
\begin{equation}
\e\sum_{t\in J_{i}\intersect\left[t_{1,}t_{2}\right]}4t^{2}g_{t}=2\e\int_{t_{1}}^{t_{2}}t\cdot g_{t}dt.\label{eq:sum_of_jumps_to_integral_main}
\end{equation}

\end{lem}

The proof is essentially a change in the order of summation, and involves
going over all points in $\left[t_{1},t_{2}\right]$ and calculating
the jump rate at each point. It is postponed to the appendix.

\begin{cor}
\label{cor:variance_from_sum_of_jumps}Let $f:\left\{ -1,1\right\} ^{n}\to\r$.
Then for all $t_{0}>0$,
\begin{equation}
\var\left(f_{t_{0}}\right)=2\e\sum_{i=1}^{n}\int_{0}^{t_{0}}t\left(\partial_{i}f_{t}\right)^{2}dt.\label{eq:variance_from_sum_of_jumps}
\end{equation}
\end{cor}

\begin{proof}
Since $f_{t_{0}}$ is a martingale, by Equation (\ref{eq:quadratic_variation_of_m}),
its variance is the expected value of the quadratic variation: 
\[
\var\left(f_{t_{0}}\right)=\e\sum_{i=1}^{n}\sum_{t\in J_{i}\intersect\left[0,t_{0}\right]}\left(2t\partial_{i}f_{t}\right){}^{2}.
\]
 Setting $g_{t}=\left(\partial_{i}f_{t}\right)^{2}$ in (\ref{eq:sum_of_jumps_to_integral_main})
completes the proof.
\end{proof}
\begin{cor}
\label{cor:differentiation_of_square}Let $f:\left\{ -1,1\right\} ^{n}\to\r$.
Then 
\[
\frac{d}{dt}\e f_{t}^{2}=2t\e\sum_{i=1}^{n}\left(\partial_{i}f_{t}\right)^{2}=2t\e\norm{\grad f_{t}}_{2}^{2}.
\]
\end{cor}

\begin{proof}
By the martingale property of $f_{t}$, 
\begin{align*}
\frac{d}{dt}\e f_{t}^{2} & =\frac{d}{dt}\left(\e f_{t}^{2}-\e f_{0}^{2}\right)=\frac{d}{dt}\left(\e\left(f_{t}-f_{0}\right)^{2}\right)=\frac{d}{dt}\var\left(f_{t}\right).
\end{align*}
Taking the derivative of equation (\ref{eq:variance_from_sum_of_jumps})
and using the fundamental theorem of calculus on the right hand side
gives the desired result.
\end{proof}
It is a basic fact (see e.g \cite[section 4.3]{garban_steif_noise_sensitivity_and_percolation})
that if $f$ has Fourier expansion $f\left(x\right)=\sum_{S}\hat{f}\left(S\right)\chi_{S}\left(x\right)$,
its noise stability is given by
\[
S_{\eps}\left(f\right)=\sum_{S\neq\emptyset}\hat{f}\left(S\right)^{2}\left(1-\eps\right)^{\abs S}.
\]
On the other hand, recalling $f_{t}=f\left(B_{t}\right)$, a short
calculation reveals that 
\[
\var\left(f_{t}\right)=\sum_{S\neq\emptyset}\hat{f}\left(S\right)^{2}t^{2\abs S}.
\]
Thus $S_{\eps}\left(f\right)=\var\left(f_{\sqrt{1-\eps}}\right),$
and the inequality (\ref{eq:improved_keller_kindler}) in the statement
of Theorem \ref{thm:improved_keller_kindler} becomes 
\[
\var\left(f_{\sqrt{1-\eps}}\right)\leq C\var\left(f\right)\left(\sum_{i=1}^{n}\Inf_{i}\left(f\right)^{2}\right)^{c\eps}.
\]
Together with equation (\ref{eq:variance_from_sum_of_jumps}), this
turns into 
\begin{equation}
\e\sum_{i=1}^{n}\int_{0}^{\sqrt{1-\eps}}t\left(\partial_{i}f_{t}\right)^{2}dt\leq C\var\left(f\right)\left(\sum_{i=1}^{n}\Inf_{i}\left(f\right)^{2}\right)^{c\eps}.\label{eq:keller_kindler_in_nice_form}
\end{equation}
We will use this formulation rather than the original statement of
(\ref{eq:improved_keller_kindler}).

For every index $i$, let $f^{\left(i\right)}$ be the harmonic extension
of $\abs{\partial_{i}f}$, and let $f_{t}^{\left(i\right)}=f^{\left(i\right)}\left(B_{t}\right)$.
If $f$ is monotone then $f^{\left(i\right)}=\partial_{i}f$, since
the derivatives are positive, but in general, 
\begin{equation}
f^{\left(i\right)}\left(x\right)\geq\abs{\partial_{i}f\left(x\right)}\,\,\all x\in\left[-1,1\right]^{n}\label{eq:harmonic_of_absolute_is_larger_than_absolute_harmonic}
\end{equation}
by convexity. In particular, plugging (\ref{eq:harmonic_of_absolute_is_larger_than_absolute_harmonic})
into Corollary \ref{cor:variance_from_sum_of_jumps}, we have 
\begin{equation}
\var\left(f_{t_{0}}\right)\leq2\sum_{i=1}^{n}\int_{0}^{t_{0}}t\e\left(f_{t}^{\left(i\right)}\right)^{2}dt.\label{eq:var_is_bounded_by_influence_process}
\end{equation}
We call the process $f_{t}^{\left(i\right)}$ the ``influence process'',
because of how the expectation of its square relates to the influence
of $f$: Observe that by (\ref{eq:influence_by_derivative}), 
\begin{equation}
f_{0}^{\left(i\right)}=\e f^{\left(i\right)}=\e\abs{\partial_{i}f}=\Inf_{i}\left(f\right).\label{eq:influence_by_harmonic_extension_of_absolute}
\end{equation}
Thus, at time $0$, we have $\e\left(f_{0}^{\left(i\right)}\right)^{2}=\left(f_{0}^{\left(i\right)}\right)^{2}=\Inf_{i}\left(f\right)^{2}$,
while at time $1$, since $f^{\left(i\right)}\left(y\right)^{2}=f^{\left(i\right)}\left(y\right)$
for $y\in\left\{ -1,1\right\} ^{n}$, we have $\e\left(f_{1}^{\left(i\right)}\right)^{2}=\e f_{1}^{\left(i\right)}=\e f^{\left(i\right)}=\Inf_{i}\left(f\right)$.
The expected value $\e\left(f_{t}^{\left(i\right)}\right)^{2}$ increases
from $\Inf_{i}\left(f\right)^{2}$ to $\Inf_{i}\left(f\right)$ as
$t$ goes from $0$ to $1$. We denote this expected value by $\psi_{i}\left(t\right):=\e\left(f_{t}^{\left(i\right)}\right)^{2}$.
Equation (\ref{eq:var_is_bounded_by_influence_process}) then becomes
\[
\var\left(f\left(B_{t_{0}}\right)\right)\leq2\sum_{i=1}^{n}\int_{0}^{t_{0}}t\psi_{i}\left(t\right)dt.
\]
The integral $\e\int_{0}^{1}t\psi_{i}\left(t\right)dt$ may be more
easily handled using a time-change which makes $\psi_{i}\left(t\right)$
log-convex; we can then bound it by a power of the influence. For
this purpose, for $s\in\left(0,\infty\right)$, denote $\varphi_{i}\left(s\right):=\psi_{i}\left(e^{-s}\right)=\e\left(f_{e^{-s}}^{\left(i\right)}\right)^{2}$.
\begin{lem}
\label{lem:squared_functions_are_log_convex}Let $g$ be the harmonic
extension of a Boolean function, and let $h\left(s\right)=g\left(B_{e^{-s}}\right)^{2}$.
Then $h\left(s\right)$ is a log-convex function of $s$.
\end{lem}

\begin{proof}
Expanding $g$ as a Fourier polynomial, we have
\begin{align}
h\left(s\right)=\e g\left(B_{e^{-s}}\right)^{2} & =\e\left[\left(\sum_{S\subseteq\left[n\right]}\widehat{g}\left(S\right)\prod_{i\in S}\left(B_{e^{-s}}^{\left(i\right)}\right)\right)^{2}\right]\nonumber \\
 & =\sum_{S\subseteq\left[n\right]}\widehat{g}\left(S\right)^{2}\prod_{i\in S}\left(B_{e^{-s}}^{\left(i\right)}\right)^{2}\nonumber \\
 & =\sum_{S\subseteq\left[n\right]}\widehat{g}\left(S\right)^{2}e^{-2s\abs S}.\label{eq:phi_as_sum_of_exponentials}
\end{align}
This is a positive linear combination of log convex-functions $e^{-2s\abs S}$,
and is therefore also log-convex \cite[section 3.5.2]{boyd_vandenberghe_convex_optimization}.
\end{proof}
Finally, we'll need the following easy technical lemma, whose short
proof is postponed to the appendix.
\begin{lem}
\label{lem:differential_inequality}Let $g:\left[0,\infty\right)\to\left[0,\infty\right)$
be a differentiable function satisfying
\begin{equation}
g'\left(t\right)\leq C\cdot g\left(t\right)\log\frac{K}{g\left(t\right)},\label{eq:differential_inequality}
\end{equation}
where $C,K$ are some positive constants. Suppose that $g\left(0\right)\leq K/2$.
Then there exists a time $t_{0}$, which depends only on $C$ and
$K$, such that for all $t\in\left[0,t_{0}\right],$ 
\[
g\left(t\right)\leq\left(\frac{1}{K}\right)^{e^{-Ct}-1}g\left(0\right)^{e^{-Ct}}.
\]
\end{lem}

\section{Proof of the improved Talagrand's conjecture}

We prove Theorem \ref{thm:improved_talagrands_conjecture} assuming
that improved Keller-Kindler inequality (\ref{eq:keller_kindler_in_nice_form})
holds; the proof of (\ref{eq:keller_kindler_in_nice_form}) is found
in Section \ref{sec:keller_kindler_proof}. Without loss of generality,
we assume that $\e f\leq0$.

The function $g_{f}:\left\{ -1,1\right\} ^{n}\to\left[0,1\right]$
used in Theorem \ref{thm:improved_talagrands_conjecture} will be
defined by 
\[
g_{f}\left(y\right)=\e\left[\sup_{0\leq s\leq1}\frac{1+f_{s}}{2}\Bigr|B_{1}=y\right].
\]
Recall Doob's martingale inequality (see e.g \cite[Theorem 4.4.4]{durrett_probability_theory_and_examples}),
which states that if $\left(X_{t}\right)_{t=0}^{1}$ is a non-negative
martingale, then 
\[
\e\left[\left(\sup_{0\leq s\leq1}X_{s}\right)^{2}\right]\leq4\e\left[X_{1}^{2}\right].
\]
Using this inequality for the non-negative martingale $X_{t}=\frac{1+f_{t}}{2}$,
we thus have
\begin{align*}
\e g^{2} & =\e\left[\e\left[\sup_{0\leq s\leq1}\frac{1+f_{s}}{2}\Bigr|B_{1}=y\right]^{2}\right]\leq\e\left[\e\left[\sup_{0\leq s\leq1}\left(\frac{1+f_{s}}{2}\right)^{2}\Bigr|B_{1}=y\right]\right]\\
 & =\e\left[\sup_{0\leq s\leq1}\left(\frac{1+f_{s}}{2}\right)^{2}\right]\leq2\left(1+\e f\right)=2\frac{1-\left(\e f\right)^{2}}{1-\e f}\leq2\var\left(f\right).
\end{align*}
as required by the theorem.

We now turn to prove the inequality (\ref{eq:improved_talagrands_conjecture})
using $g_{f}$. To this end, we can relate the product $h_{f}^{p}\left(y\right)g_{f}\left(y\right)$
to the stochastic constructions in the previous section. By definition
of the discrete derivative, for any $y\in\left\{ -1,1\right\} ^{n}$,
$\partial_{i}f\left(y\right)=0$ if $f\left(y\right)=f\left(y^{\oplus i}\right)$,
and $\partial_{i}f\left(y\right)=\pm1$ if $f\left(y\right)\neq f\left(x^{\oplus i}\right)$,
so 
\[
h_{f}\left(y\right)=\sum_{i=1}^{n}\partial_{i}f\left(y\right)^{2}=\norm{\grad f\left(y\right)}_{2}^{2}.
\]
We thus have $h_{f}^{p}\left(y\right)=\norm{\grad f\left(y\right)}_{2}^{2p}$,
and using this relation we define the stochastic process
\[
\Psi_{t}=\norm{\grad f_{t}}_{2}^{2p}\sup_{0\leq s\leq t}\frac{1+f_{s}}{2},
\]
noting that
\[
\e\Psi_{1}=\e\left[h_{f}^{p}g_{f}\right].
\]
Our goal is therefore to bound $\e\Psi_{1}$ from below. 

A simple calculation, using the fact that $\norm{\grad f_{t}}_{2}^{2p}$
is a submartingale, shows that $\e\Psi_{t}$ is an increasing function
of $t$. For the rest of this section, we therefore assume that 
\begin{equation}
\e\Psi_{t}\leq\var\left(f\right)\cdot\left(\log\left(2+\frac{e}{\sum_{i}\Inf_{i}\left(f\right)^{2}}\right)\right)^{p}\,\,\,\,;0\leq t\leq1,\label{eq:psi_is_small_assumption}
\end{equation}
otherwise there is nothing to prove. 

Our proof relies on the existence of a stopping time $\tau_{\alpha}$
such that with high probability $\Psi_{\tau_{\alpha}}$ is large,
as is shown by the following proposition. Fix $\alpha>0$ whose value
is to be chosen later, and define 
\[
\tau_{\alpha}=\inf\left\{ 0\leq t\leq1\Bigr|\Psi_{t}>\frac{1}{8}\alpha\left(\log\left(2+\frac{e}{\sum_{i}\Inf_{i}\left(f\right)^{2}}\right)\right)^{p}\right\} \land1.
\]

\begin{prop}
\label{prop:stopping_time_has_large_probability}Assume that (\ref{eq:psi_is_small_assumption})
holds. Then there exists an $\alpha>0$ such that
\[
\p\left[\tau_{\alpha}<1\right]\geq C\var\left(f\right),
\]
where $C>0$ is a constant which depends only on $\alpha$. 
\end{prop}

Assuming the above proposition holds, the improved Talagrand's conjecture
swiftly follows:
\begin{proof}[Proof of Theorem \ref{thm:improved_talagrands_conjecture}]
 By conditioning on the event $\tau_{\alpha}<1$, we have 
\begin{align*}
\e\Psi_{1} & \geq\e\Psi_{\tau_{\alpha}}\geq\e\left[\Psi_{\tau_{\alpha}}\Bigr|\tau_{\alpha}<1\right]\p\left[\tau_{\alpha}<1\right]\\
 & \geq\frac{1}{8}\alpha\left(\log\left(2+\frac{e}{\sum_{i}\Inf_{i}\left(f\right)^{2}}\right)\right)^{p}\cdot C\var\left(f\right).
\end{align*}
\end{proof}
The rest of this section is devoted to proving Proposition \ref{prop:stopping_time_has_large_probability}.
The main idea is to see how different sample paths contribute to the
quadratic variation of $f_{t}$. On the one hand, the lion's share
of the quadratic variation is gained from paths where the gradient's
norm $\norm{\grad f_{t}}_{2}$ is large. On the other hand, the quadratic
variation has a relatively high probability to be large, and so $\norm{\grad f_{t}}_{2}$
must be large with relatively high probability as well. This argument
takes care of the gradient's contribution to $\Psi_{t}$; to deal
with the supremum's contribution, we show that with high enough probability,
either $f_{t}$ makes a large jump (which causes both $\sup f_{t}$
and the gradient's norm to be large at the same time) or there is
a time where $f_{t}$'s position is bounded away from the endpoints
$\left\{ -1,1\right\} $, allowing its gradient to be large later
on. 

\begin{proof}[Proof of Proposition \ref{prop:stopping_time_has_large_probability}]Let
$\theta=\inf\left\{ t\geq0\mid f_{t}>0\right\} \land1$. Since $f_{t}$
is a martingale, 
\[
f_{0}=\e f_{1}=2\p\left[f_{1}=1\right]-1,
\]
and since $\left\{ f_{1}=1\right\} \subseteq\left\{ \theta<1\right\} $,
\begin{equation}
\p\left[\theta<1\right]\geq\p\left[f_{1}=1\right]=\frac{1+f_{0}}{2}=\frac{1-f_{0}^{2}}{2\left(1-f_{0}\right)}=\frac{\var\left(f\right)}{2\left(1-f_{0}\right)}\stackrel{\left(f_{0}\leq0\right)}{\geq}\frac{1}{4}\var\left(f\right).\label{eq:prob_that_theta_occurs}
\end{equation}
By conditioning on $\theta<1$ we have 
\begin{align*}
\p\left[\tau_{\alpha}<1\right] & \geq\p\left[\tau_{\alpha}<1\Bigr|\theta<1\right]\p\left[\theta<1\right]\\
 & \geq\frac{1}{4}\var\left(f\right)\cdot\p\left[\tau_{\alpha}<1\Bigr|\theta<1\right].
\end{align*}
It remains only to show that for small enough (but fixed) $\alpha$,
$\p\left[\tau_{\alpha}<1\mid\theta<1\right]$ is larger than some
constant. 

Since we assume $\e f\leq0$, the process $f_{t}$ starts at $f_{0}\leq0$.
There are two different ways for $f_{t}$ to cross above the value
$0$: It could either move across it continuously, or it could jump
from some value smaller than $0$ to some value larger than $0$.
We now divide the analysis into two cases, depending on the probability
that $f_{t}$ makes a very large jump over $0$ at time $\theta$. 

\subsection*{Case 1: Small jump}

Suppose that $\p\left[f_{\theta}\in\left[0,1/2\right]\mid\theta<1\right]\geq1/2$,
i.e with constant probability, no large jump was made at time $\theta$.
The next two propositions show that with high probability the quadratic
variation gained from time $\theta$ onward must be large, and that
most of the quadratic variation is gained at times when, just before
the process jumps, the gradient is large. To make these notions precise,
we will need the following definitions. Let 
\[
F_{\alpha,t}=\left\{ \norm{\grad f_{t^{-}}}_{2}>\alpha\sqrt{\log\left(2+\frac{e}{\sum_{i}\Inf_{i}\left(f\right)^{2}}\right)}\right\} 
\]
be the event that the norm of the gradient is large just before time
$t$; let 
\[
E_{t}=\left\{ \sup_{0\leq s<t}f_{s}\geq0\right\} ,
\]
be the event that $\sup f_{s}$ is large strictly before time $t$;
let
\[
V_{\alpha}=\sum_{i=1}^{n}\sum_{t\in J_{i}\intersect\left[0,1\right]}\left(2t\partial_{i}f_{t}\right)^{2}\one_{F_{\alpha,t}^{C}}\one_{E_{t}},
\]
be the quadratic variation accumulated at times when the supremum
is large but the gradient is small; and let 
\[
V^{t_{1}\to t_{2}}=\sum_{i=1}^{n}\sum_{t\in J_{i}\intersect\left(t_{1},t_{2}\right]}\left(2t\partial_{i}f_{t}\right)^{2}
\]
be the gain in quadratic variation from time $t_{1}$ up to and including
time $t_{2}$. 
\begin{prop}
\label{prop:left_quadratic_variation_of_small_gradient_is_small}Let
$0<\alpha<1/e$, and assume that (\ref{eq:psi_is_small_assumption})
holds. There exists a function $\rho:\left[0,1\right]\to\r$ with
$\lim_{x\to0}\rho\left(x\right)=0$ such that 
\begin{equation}
\e\left[V_{\alpha}\right]\leq\rho\left(\alpha\right)\var\left(f\right).\label{eq:left_quadratic_variation_of_small_gradient_is_small}
\end{equation}
\end{prop}

\begin{proof}
We first express $\e\left[V_{\alpha}\right]=\e\left[\sum_{i=1}^{n}\sum_{t\in J_{i}\intersect\left[0,1\right]}\left(2t\partial_{i}f_{t}\right)^{2}\one_{F_{\alpha}^{C}}\one_{E_{t}}\right]$
as an integral, rather than a sum over jumps. Since $\partial_{i}f_{t}$
is independent of coordinate $i$, we have that for $t\in J_{i}$,
$\partial_{i}f_{t}=\partial_{i}f_{t^{-}}$. Thus 
\[
\e\left[V_{\alpha}\right]=\e\left[\sum_{i=1}^{n}\sum_{t\in J_{i}\intersect\left[0,1\right]}\left(2t\partial_{i}f_{t^{-}}\right)^{2}\one_{F_{\alpha}^{C}}\one_{E_{t}}\right].
\]
The process $g_{t}=\left(\partial_{i}f_{t^{-}}\right)^{2}\one_{F_{\alpha}^{C}}\one_{E_{t}}$
is measurable with respect to the filtration generated by $\left\{ B_{s}\right\} _{0\leq s<t}$
and is left-continuous. Invoking Lemma \ref{lem:sum_of_jumps_to_integral},
we have
\begin{align}
\e\left[V_{\alpha}\right] & =\e\sum_{i=1}^{n}\sum_{t\in J_{i}\intersect\left[0,1\right]}4t^{2}g_{t}\nonumber \\
\left(\text{Lemma \ref{lem:sum_of_jumps_to_integral}}\right) & =2\int_{0}^{1}t\e\left[\norm{\grad f_{t^{-}}}_{2}^{2}\one_{F_{\alpha}^{C}}\one_{E_{t}}\right]dt\nonumber \\
 & =2\int_{0}^{1}t\e\left[\norm{\grad f_{t}}_{2}^{2}\one_{F_{\alpha}^{C}}\one_{E_{t}}\right]dt,\label{eq:sum_to_integral_of_jumps_with_gradient}
\end{align}
where the last equality is because $\norm{\grad f_{t^{-}}}_{2}^{2}\one_{F_{\alpha}^{C}}\one_{E_{t}}$
can differ from $\norm{\grad f_{t}}_{2}^{2}\one_{F_{\alpha}^{C}}\one_{E_{t}}$
only at discontinuities.

Let $\delta'>0$ be defined as 
\[
\delta'=\begin{cases}
\frac{c^{-1}\log\left(1/\alpha\right)}{\log\left(2+\frac{e}{\sum_{i}\Inf_{i}\left(f\right)^{2}}\right)} & \sum_{i}\Inf_{i}\left(f\right)^{2}\leq\frac{1}{2}\\
1 & \text{otherwise},
\end{cases}
\]
where $c$ is the universal constant from Theorem \ref{thm:improved_keller_kindler},
and set 
\[
\delta=\min\left\{ \delta',1\right\} .
\]
Consider the integral 
\[
\int_{0}^{1}t\e\left[\norm{\grad f_{t}}_{2}^{2}\one_{F_{\alpha,t}^{C}}\one_{E_{t}}\right]dt=\int_{0}^{1-\delta}t\e\left[\norm{\grad f_{t}}_{2}^{2}\one_{F_{\alpha,t}^{C}}\one_{E_{t}}\right]dt+\int_{1-\delta}^{1}t\e\left[\norm{\grad f_{t}}_{2}^{2}\one_{F_{\alpha,t}^{C}}\one_{E_{t}}\right]dt.
\]
The first integral on the right hand side is equal to $0$ if $\delta=1$.
Otherwise, we necessarily have that $\sum_{i}\Inf_{i}\left(f\right)^{2}\leq1/2$,
in which case the integral can be bounded using equation (\ref{eq:keller_kindler_in_nice_form}):
Since $1-\delta\leq\sqrt{1-\delta}$ for all $\delta\in\left[0,1\right]$,
we have

\begin{align*}
\int_{0}^{1-\delta}t\e\left[\norm{\grad f_{t}}_{2}^{2}\one_{F_{\alpha,t}^{C}}\one_{E_{t}}\right]dt & \leq\e\int_{0}^{1-\delta}t\norm{\grad f_{t}}_{2}^{2}dt\\
 & \leq\e\int_{0}^{\sqrt{1-\delta}}t\norm{\grad f_{t}}_{2}^{2}dt\\
 & \stackrel{\left(\ref{eq:keller_kindler_in_nice_form}\right)}{\leq}C_{1}\var\left(f\right)\left(\sum_{i}\Inf_{i}\left(f\right)^{2}\right)^{c\delta}\\
 & \leq C_{1}\var\left(f\right)\alpha^{\log\left(\sum_{i}\Inf_{i}\left(f\right)^{2}\right)\log\left(\frac{\sum_{i}\Inf_{i}\left(f\right)^{2}}{2\sum_{i}\Inf_{i}\left(f\right)^{2}+e}\right)^{-1}}
\end{align*}
for some constant $C_{1}>0$. Since $\sum_{i}\Inf_{i}\left(f\right)^{2}\leq1/2$,
the exponent $\log\left(\sum_{i}\Inf_{i}\left(f\right)^{2}\right)\log\left(\frac{\sum_{i}\Inf_{i}\left(f\right)^{2}}{2\sum_{i}\Inf_{i}\left(f\right)^{2}+e}\right)^{-1}$
is bounded below by $1/3$. Since $\alpha<1$, we thus have that regardless
of the value of $\sum_{i}\Inf_{i}\left(f\right)^{2}$,
\begin{equation}
\int_{0}^{1-\delta}t\e\left[\norm{\grad f_{t}}_{2}^{2}\one_{F_{\alpha,t}^{C}}\one_{E_{t}}\right]dt\leq C_{1}\alpha^{1/3}\var\left(f\right).\label{eq:tal_truncated_1-1}
\end{equation}
The second integral on the right hand side can be bounded using the
definitions of $F_{\alpha,t}$ and $E_{t}$: By definition of $F_{\alpha,t}$
we have

\begin{align}
\e\left[\norm{\grad f_{t^{-}}}_{2}^{2-2p}\norm{\grad f_{t}}_{2}^{2p}\one_{F_{\alpha,t}^{C}}\one_{E_{t}}\right] & =\e\left[\left(\norm{\grad f_{t^{-}}}_{2}^{2-2p}\one_{F_{\alpha,t}^{C}}\right)\norm{\grad f_{t}}_{2}^{2p}\one_{E_{t}}\right]\nonumber \\
 & \leq\alpha^{2-2p}\left(\log\left(2+\frac{e}{\sum_{i}\Inf_{i}\left(f\right)^{2}}\right)\right)^{1-p}\e\left[\norm{\grad f_{t}}_{2}^{2p}\one_{E_{t}}\right],\label{eq:using _def_of_f_alpha_t}
\end{align}
whereas by definition of $E_{t}$, 
\begin{align*}
\e\left[\norm{\grad f_{t}}_{2}^{2p}\one_{E_{t}}\right] & \leq2\e\left[\norm{\grad f_{t}}_{2}^{2p}\sup_{s<t}\frac{1+f_{s}}{2}\one_{E_{t}}\right]\\
 & \leq2\e\left[\norm{\grad f_{t}}_{2}^{2p}\sup_{s<t}\frac{1+f_{s}}{2}\right]\\
 & \leq2\Psi\left(t\right)\leq2\var\left(f\right)\left(\log\left(2+\frac{e}{\sum_{i}\Inf_{i}\left(f\right)^{2}}\right)\right)^{p}.
\end{align*}
Plugging the above display into (\ref{eq:using _def_of_f_alpha_t}),
the integral can be bounded by 
\begin{align*}
\int_{1-\delta}^{1}t\e\left[\norm{\grad f_{t}}_{2}^{2}\one_{F_{\alpha,t}^{C}}\one_{E_{t}}\right]dt & =\int_{1-\delta}^{1}t\e\left[\norm{\grad f_{t^{-}}}_{2}^{2-2p}\norm{\grad f_{t}}_{2}^{2p}\one_{F_{\alpha,t}^{C}}\one_{E_{t}}\right]dt\\
 & \leq\delta\alpha^{2-2p}\var\left(f\right)\log\left(2+\frac{e}{\sum_{i}\Inf_{i}\left(f\right)^{2}}\right).
\end{align*}
Since in any case $\delta\leq C_{2}\cdot\frac{c^{-1}\log\left(1/\alpha\right)}{\log\left(2+\frac{e}{\sum_{i}\Inf_{i}\left(f\right)^{2}}\right)}$
for some constant $C_{2}>0$, we thus have 
\begin{equation}
\int_{1-\delta}^{1}t\e\left[\norm{\grad f_{t}}_{2}^{2}\one_{F_{\alpha,t}^{C}}\one_{E_{t}}\right]dt\leq C_{2}c^{-1}\alpha^{2-2p}\log\left(1/\alpha\right)\var\left(f\right).\label{eq:tal_truncated_2-1}
\end{equation}
Combining (\ref{eq:tal_truncated_1-1}) and (\ref{eq:tal_truncated_2-1}),
there exists an absolute constant $C:=C_{1}+C_{2}c^{-1}>0$ such that
\begin{equation}
\int_{0}^{1}t\e\left(\norm{\grad f_{t}}_{2}^{2}\one_{F_{\alpha,t}^{C}}\one_{E_{t}}\right)dt\leq C\left(\alpha^{1/3}+\alpha^{2-2p}\right)\log\left(1/\alpha\right)\var\left(f\right).\label{eq:tal_truncated_3-1}
\end{equation}
Plugging this into (\ref{eq:sum_to_integral_of_jumps_with_gradient})
finishes the proof, with $\rho\left(x\right)=C\left(x^{1/3}+x^{2-2p}\right)\log\left(1/x\right)$.
\end{proof}

\begin{prop}
\label{prop:quadratic_variation_is_not_concentrated_near_zero}Let
$a\in\left[0,1\right)$ and let $0\leq\theta\leq1$ be a $B_{t}$-measurable
stopping time such that
\[
\p\left[f_{\theta}\in\left[-a,a\right]\Bigr|\theta<1\right]\geq q
\]
for some $q\in\left[0,1\right]$. Then 
\[
\p\left[V^{\theta\to1}\geq\frac{1}{5}q\left(1-a\right)^{2}\Bigr|\theta<1\right]\geq\frac{1}{9}q\left(1-a\right)^{2}.
\]
\end{prop}

This proposition reflects the intuition that for a martingale to reach
a point far from its initial position, it should have a large quadratic
variation. The proof, however, requires using the particular details
of the way the martingale jumps.

\begin{proof}
Let $x>0$ be a number whose value will be chosen later, and let $\sigma=\inf\left\{ t>\theta\mid V^{\theta\to t}\geq x^{2}\right\} \land1$
be the first time that the quadratic variation grows beyond $x^{2}$.
Since the quadratic variation increases only when $f_{t}$ jumps,
and since the probability of jumping at time $t=1$ is $0$, the event
$\left\{ V^{\theta\to1}\geq x^{2}\mid\theta<1\right\} $ is equal
to the event $\left\{ \sigma<1\mid\theta<1\right\} $, which in turn
implies that $\abs{f_{\sigma}}<1$: 
\begin{equation}
\p\left[V^{\theta\to1}\geq x^{2}\Bigr|\theta<1\right]\geq\p\left[\abs{f_{\sigma}}<1\Bigr|\theta<1\right].\label{eq:quadratic_variation_in_talagrand_conjecture-1}
\end{equation}
Hence it suffices to bound the probability that $\abs{f_{\sigma}}<1$.
This can be done by looking at the second moment $\e\left(f_{\sigma}-f_{\theta}\right)^{2}$:
On one hand, since $f_{\theta}\in\left[-a,a\right]$ with probability
at least $q$ when $\theta<1$, we have
\begin{align*}
\e\left[\left(f_{\sigma}-f_{\theta}\right)^{2}\Bigr|\theta<1\right] & \geq\e\left[\left(f_{\sigma}-f_{\theta}\right)^{2}\mid\left\{ \theta<1\right\} \intersect\left\{ \abs{f_{\sigma}}=1\right\} \right]\p\left[\abs{f_{\sigma}}=1\Bigr|\theta<1\right]\\
 & \geq q\left(1-a\right)^{2}\p\left[\abs{f_{\sigma}}=1\Bigr|\theta<1\right],
\end{align*}
giving 
\begin{equation}
\p\left[\abs{f_{\sigma}}<1\Bigr|\theta<1\right]\geq1-\frac{\e\left[\left(f_{\sigma}-f_{\theta}\right)^{2}\Bigr|\theta<1\right]}{q\left(1-a\right)^{2}}.\label{eq:quadratic_variation_in_talagrand_conjecture_2-1}
\end{equation}
On the other hand, $\e\left[\left(f_{\sigma}-f_{\theta}\right)^{2}\mid\theta<1\right]$
can be bounded by considering the size of the jumps of $f_{t}$: Since
the $\sigma$-algebra generated by the event $\theta<1$ is contained
in that generated by $B_{\theta}$,
\[
\e\left[\left(f_{\sigma}-f_{\theta}\right)^{2}\Bigr|\theta<1\right]=\e\left[\e\left[\left(f_{\sigma}-f_{\theta}\right)^{2}\mid B_{\theta}\right]\Bigr|\theta<1\right].
\]
Conditioned on $B_{\theta}$, the process $f_{t}$ is a martingale
for $t\in\left[\theta,\sigma\right]$, and so by (\ref{eq:variance_of_martingales_via_quadratic_variation}),
\begin{align*}
\e\left[\e\left[\left(f_{\sigma}-f_{\theta}\right)^{2}\mid B_{\theta}\right]\Bigr|\theta<1\right] & =\e\left[\e\left[\left(\left[f\right]_{\sigma}-\left[f\right]_{\theta}\right)\mid B_{\theta}\right]\Bigr|\theta<1\right]\\
 & =\e\left[\left[f\right]_{\sigma}-\left[f\right]_{\theta}\Bigr|\theta<1\right]\\
 & =\e\left[V^{\theta\to\sigma}\Bigr|\theta<1\right]\\
 & =\e\left[V^{\theta\to\sigma^{-}}+\left(\Delta f_{\sigma}\right)^{2}\Bigr|\theta<1\right].
\end{align*}
By definition of $\sigma$, $V^{\theta\to\sigma^{-}}\leq x^{2}$,
and since all jumps are bounded by $2$,
\begin{align}
\e\left[\left(f_{\sigma}-f_{\theta}\right)^{2}\Bigr|\theta<1\right] & \leq x^{2}+\e\left[\left(\Delta f_{\sigma}\right)^{2}\Bigr|\theta<1\right]\nonumber \\
 & \leq x^{2}+x^{2}\p\left[\Delta f_{\sigma}<x\Bigr|\theta<1\right]+4\cdot\p\left[\Delta f_{\sigma}\geq x\Bigr|\theta<1\right]\nonumber \\
 & =x^{2}+x^{2}\left(1-\p\left[\Delta f_{\sigma}\geq x\Bigr|\theta<1\right]\right)+4\p\left[\Delta f_{\sigma}\geq x\Bigr|\theta<1\right]\nonumber \\
 & =2x^{2}+\p\left[\Delta f_{\sigma}\geq x\Bigr|\theta<1\right]\left(4-x^{2}\right).\label{eq:small_jumps_in_talagrand_conjecture_2-1}
\end{align}
Plugging this into (\ref{eq:quadratic_variation_in_talagrand_conjecture-1})
and (\ref{eq:quadratic_variation_in_talagrand_conjecture_2-1}), we
get
\[
\p\left[V^{\theta\to1}\geq x^{2}\Bigr|\theta<1\right]\geq1-\frac{2x^{2}+\p\left[\Delta f_{\sigma}\geq x\Bigr|\theta<1\right]\left(4-x^{2}\right)}{q\left(1-a\right)^{2}}.
\]
Since $\left\{ \Delta f_{\sigma}\geq x\mid\theta>1\right\} \subseteq\left\{ V^{\theta\to1}\geq x^{2}\mid\theta>1\right\} $,
this gives
\[
\p\left[V^{\theta\to1}\geq x^{2}\Bigr|\theta<1\right]\geq1-\frac{2x^{2}+\p\left[V^{\theta\to1}\geq x^{2}\Bigr|\theta<1\right]\left(4-x^{2}\right)}{q\left(1-a\right)^{2}}.
\]
Solving for $\p\left[V^{\theta\to1}\geq x^{2}\mid\theta<1\right]$,
we have
\begin{align*}
\p\left[V^{\theta\to1}\geq x^{2}\Bigr|\theta<1\right] & \geq1-\frac{4+x^{2}}{q\left(1-a\right)^{2}+4-x^{2}}.
\end{align*}
 In particular, for $x^{2}=\frac{4q\left(1-a\right)^{2}}{16+p\left(1-a\right)^{2}}\geq\frac{1}{5}q\left(1-a\right)^{2}$,
we get
\[
\p\left[V^{\theta\to1}\geq\frac{1}{5}q\left(1-a\right)^{2}\Bigr|\theta<1\right]\geq\frac{1}{9}q\left(1-a\right)^{2}.
\]

\end{proof}

With the above two propositions, we can show that $\p\left[\tau_{\alpha}<1\mid\theta<1\right]\geq C\var\left(f\right)$.
On one hand, by Proposition \ref{prop:left_quadratic_variation_of_small_gradient_is_small},
the expected increase in the quadratic variation from time $\theta$
to time $1$ of paths with small gradient must be small: Denoting
$A_{\alpha,\theta}=\bigcap_{t\in\left(\theta,1\right]}F_{\alpha,t}^{c}$
to be the event that the gradient was small at all times from $\theta$
to $1$, we have

\begin{align}
\e\left[V^{\theta\to1}\one_{A_{\alpha,\theta}}\Bigr|\theta<1\right] & \leq\e\left[\sum_{i}\sum_{t\in J_{i}\intersect\left(\theta,1\right]}\left(2\partial_{i}f_{t}\right)^{2}\one_{F_{\alpha,t}^{C}}\Bigr|\theta<1\right]\label{eq:intermediate_quadratic}\\
\left(\text{since \ensuremath{\theta<1}}\right) & =\e\left[\sum_{i}\sum_{t\in J_{i}\intersect\left(\theta,1\right]}\left(2\partial_{i}f_{t}\right)^{2}\one_{F_{\alpha,t}^{C}}\one_{E_{t}}\Bigr|\theta<1\right]\nonumber \\
 & =\frac{\e\left[\sum_{i}\sum_{t\in J_{i}\intersect\left(\theta,1\right]}\left(2\partial_{i}f_{t}\right)^{2}\one_{F_{\alpha,t}^{C}}\one_{E_{t}}\right]}{\p\left[\theta<1\right]}\nonumber \\
 & \leq\frac{\e\left[V_{\alpha}\right]}{\p\left[\theta<1\right]}\stackrel{\left(\ref{eq:left_quadratic_variation_of_small_gradient_is_small}\right),\left(\ref{eq:prob_that_theta_occurs}\right)}{\leq}4\rho\left(\alpha\right).\nonumber 
\end{align}
On the other hand, by invoking Proposition \ref{prop:quadratic_variation_is_not_concentrated_near_zero}
with $q=1/2$ and $a=1/2$, the overall increase in quadratic variation
is large with high probability:
\[
\p\left[V^{\theta\to1}>\frac{1}{40}\Bigr|\theta<1\right]\geq\frac{1}{72}.
\]
We then have
\begin{align*}
\e\left[V^{\theta\to1}\one_{A_{\alpha,\theta}}\Bigr|\theta<1\right] & \geq\frac{1}{40}\p\left[\left\{ V^{\theta\to1}>\frac{1}{40}\right\} \intersect A_{\alpha,\theta}\Bigr|\theta<1\right]\\
 & \geq\frac{1}{40}\left(\p\left[V^{\theta\to1}>\frac{1}{40}\Bigr|\theta<1\right]-\p\left[A_{\alpha,\theta}^{C}\Bigr|\theta<1\right]\right).
\end{align*}
Solving for $\p\left[A_{\alpha,\theta}^{C}\mid\theta<1\right]$ gives
\begin{align*}
\p\left[A_{\alpha,\theta}^{C}\Bigr|\theta<1\right] & \geq\p\left[V^{\theta\to1}>\frac{1}{40}\Bigr|\theta<1\right]-40\e\left[V^{\theta\to1}\one_{A_{\alpha,\theta}}\Bigr|\theta<1\right]\\
 & \geq\frac{1}{72}-C\rho\left(\alpha\right).
\end{align*}
Since $\lim_{\alpha\to0}\rho\left(\alpha\right)=0$, this last expression
is larger than some constant $c$ for small enough $\alpha$. But
the event $A_{\alpha,\theta}^{C}$ means that at some time $t^{*}>\theta$,
the gradient $\norm{\grad f_{t^{*}}}_{2}$ was larger than $\alpha\sqrt{\log\left(2+\frac{e}{\sum_{i}\Inf_{i}\left(f\right)^{2}}\right)}$,
while the event $\theta<1$ implies that $\sup_{0\leq s\leq\theta}\frac{1+f_{s}}{2}\geq1/2$,
yielding $\Psi_{t^{*}}\geq\frac{1}{2}\alpha\left(\log\left(2+\frac{e}{\sum_{i}\Inf_{i}\left(f\right)^{2}}\right)\right)^{p}$.
Thus
\[
\p\left[\tau_{\alpha}<1\Bigr|\theta<1\right]\geq\p\left[A_{\alpha,\theta}^{C}\Bigr|\theta<1\right]>c
\]
as needed.

\subsection*{Case 2: Large jump}

Suppose now that $\p\left[f_{\theta}\in\left[0,\frac{1}{2}\right]\mid\theta<1\right]<1/2,$
meaning that with large probability $f_{t}$ makes a large jump at
time $\theta$ to some value greater than $1/2$:
\begin{equation}
\p\left[\Delta f_{\theta}\geq\frac{1}{2}\Bigr|\theta<1\right]\geq\frac{1}{2}.\label{eq:case_large_jump}
\end{equation}
The next proposition, parallel to Proposition \ref{prop:left_quadratic_variation_of_small_gradient_is_small},
shows that most of the quadratic variation is gained at times when,
just after the process jumped, the gradient was large. For a fixed
$\alpha>0$ whose value is to be chosen later, let

\[
H_{\alpha,t}=\left\{ \norm{\grad f_{t}}_{2}>\alpha\sqrt{\log\left(2+\frac{e}{\sum_{i}\Inf_{i}\left(f\right)^{2}}\right)}\right\} 
\]
be the event that the norm of the gradient is large exactly at time
$t$, and let
\[
U_{\alpha}=\sum_{i=1}^{n}\sum_{t\in J_{i}\intersect\left[0,1\right]}\left(2t\partial_{i}f_{t}\right)^{2}\one_{H_{\alpha,t}^{C}}\one_{\left\{ f_{t}\geq0\right\} }
\]
be the quadratic variation accumulated at times when $f$'s value
is large but the gradient is small.
\begin{prop}
\label{prop:right_quadratic_variation_of_small_gradient_is_small}Let
$0<\alpha<1/e$, and assume that (\ref{eq:psi_is_small_assumption})
holds. There exists a function $\rho:\left[0,1\right]\to\r$ with
$\lim_{x\to0}\rho\left(x\right)=0$ such that 
\begin{equation}
\e\left[U_{\alpha}\right]\leq\rho\left(\alpha\right)\var\left(f\right).\label{eq:right_quadratic_variation_of_small_gradient_is_small}
\end{equation}
\end{prop}

\begin{proof}
The proof is somewhat similar to that of Proposition \ref{prop:left_quadratic_variation_of_small_gradient_is_small}.
Observe that the random variable $\one_{H_{\alpha,t}^{C}}\one_{\left\{ f_{t}\geq0\right\} }$
is a function only of $B_{t}$; there is therefore a continuous ``interpolating''
function $h:\left[-1,1\right]^{n}\to\r$ such that 
\[
\one_{H_{\alpha,t}^{C}}\one_{\left\{ f_{t}\geq0\right\} }\leq h\left(B_{t}\right)\leq\one_{H_{2\alpha,t}^{c}}\one_{\left\{ f_{t}\geq-1/2\right\} }.
\]
Invoking Lemma \ref{lem:sum_of_jumps_to_integral} with $g_{t}=\left(\partial_{i}f_{t}\right)^{2}h\left(B_{t}\right)$,
we have
\begin{align}
\e\left[U_{\alpha}\right] & =\e\left[\sum_{i=1}^{n}\sum_{t\in J_{i}\intersect\left[0,1\right]}\left(2t\partial_{i}f_{t}\right)^{2}\one_{H_{\alpha,t}^{C}}\one_{\left\{ f_{t}\geq0\right\} }\right]\leq\e\left[\sum_{i=1}^{n}\sum_{t\in J_{i}\intersect\left[0,1\right]}4t^{2}g_{t}\right]\nonumber \\
\left(\ref{eq:sum_of_jumps_to_integral_main}\right) & =2\int_{0}^{1}t\e\left[\norm{\grad f_{t}}_{2}^{2}h\left(B_{t}\right)\right]dt\nonumber \\
 & \leq2\int_{0}^{1}t\e\left[\norm{\grad f_{t}}_{2}^{2}\one_{H_{2\alpha,t}^{c}}\one_{\left\{ f_{t}\geq-1/2\right\} }\right]dt.\label{eq:right_sum_to_integral_of_jumps_with_gradient}
\end{align}
We now define $\delta$ as in Proposition \ref{prop:left_quadratic_variation_of_small_gradient_is_small},
and split the integral in two: 
\[
\int_{0}^{1}\left(\ldots\right)dt=\int_{0}^{1-\delta}\left(\ldots\right)dt+\int_{1-\delta}^{1}\left(\ldots\right)dt.
\]
The first integral on the right hand side is dealt with exactly as
in Proposition \ref{prop:left_quadratic_variation_of_small_gradient_is_small},
yielding
\begin{equation}
\int_{0}^{1-\delta}t\e\left[\norm{\grad f_{t}}_{2}^{2}\one_{H_{\alpha,t}^{C}}\one_{\left\{ f_{t}\geq-1/2\right\} }\right]dt\leq C_{1}\alpha^{1/3}\var\left(f\right).\label{eq:right_tal_truncated}
\end{equation}
For the second integral on the right hand side, we again use the fact
that $\Psi_{t}$ is bounded: By definition of $H_{2\alpha,t}$,
\begin{align}
\e\left[\norm{\grad f_{t}}_{2}^{2}\one_{H_{2\alpha,t}^{c}}\one_{\left\{ f_{t}\geq-1/2\right\} }\right] & =\e\left[\left(\norm{\grad f_{t}}_{2}^{2-2p}\one_{H_{2\alpha,t}^{C}}\right)\norm{\grad f_{t}}_{2}^{2p}\one_{\left\{ f_{t}\geq-1/2\right\} }\right]\nonumber \\
 & \leq\e\left[\left(2\alpha\right)^{2-2p}\left(\log\left(2+\frac{e}{\sum_{i}\Inf_{i}\left(f\right)^{2}}\right)\right)^{1-p}\norm{\grad f_{t}}_{2}^{2p}\one_{\left\{ f_{t}\geq-1/2\right\} }\right],\label{eq:using _def_of_h_alpha_t}
\end{align}
whereas
\[
\e\left[\norm{\grad f_{t}}_{2}^{2p}\one_{\left\{ f_{t}\geq-1/2\right\} }\right]\leq4\e\left[\norm{\grad f_{t}}_{2}^{2p}\sup_{s\leq t}\frac{1+f_{s}}{2}\right]\leq4\Psi_{t}\leq4\var\left(f\right)\left(\log\left(2+\frac{e}{\sum_{i}\Inf_{i}\left(f\right)^{2}}\right)\right)^{p}.
\]
Plugging the above display into (\ref{eq:using _def_of_h_alpha_t}),
we therefore have 
\[
\int_{1-\delta}^{1}t\e\left[\norm{\grad f_{t}}_{2}^{2}\one_{H_{2\alpha,t}^{C}}\one_{\left\{ f\geq-1/2\right\} }\right]dt\leq16\delta\alpha^{2-2p}\var\left(f\right)\log\left(2+\frac{e}{\sum_{i}\Inf_{i}\left(f\right)^{2}}\right).
\]
Since in any case $\delta\leq C_{2}\cdot\frac{c^{-1}\log\left(1/\alpha\right)}{\log\left(2+\frac{e}{\sum_{i}\Inf_{i}\left(f\right)^{2}}\right)}$
for some constant $C_{2}>0$, we thus have 
\begin{equation}
\int_{1-\delta}^{1}t\e\left[\norm{\grad f_{t}}_{2}^{2}\one_{H_{2\alpha,t}^{C}}\one_{\left\{ f\geq-1/2\right\} }\right]dt\leq16C_{2}c^{-1}\alpha^{2-2p}\log\left(1/\alpha\right)\var\left(f\right).\label{eq:right_tal_truncated_2}
\end{equation}
Combining (\ref{eq:right_tal_truncated}) and (\ref{eq:right_tal_truncated_2}),
there exists an absolute constant $C:=C_{1}+16C_{2}c^{-1}>0$ such
that 
\[
\int_{0}^{1}t\e\left[\norm{\grad f_{t}}_{2}^{2}\one_{H_{2\alpha,t}^{C}}\one_{\left\{ f_{t}\geq-1/2\right\} }\right]dt\leq C\left(\alpha^{1/3}+\alpha^{2-2p}\right)\log\left(1/\alpha\right)\var\left(f\right).
\]
Plugging this into (\ref{eq:right_sum_to_integral_of_jumps_with_gradient})
finishes the proof with $\rho\left(x\right)=C\left(x^{1/3}+x^{2-2p}\right)\log\left(1/x\right)$.
\end{proof}

With this proposition in hand, we can show that $\p\left[\tau_{\alpha}<1\right]\geq C\var\left(f\right)$.
Consider the event that at time $\theta$, both $\norm{\grad f_{\theta}}_{2}<\alpha\sqrt{\log\left(2+\frac{e}{\sum_{i}\Inf_{i}\left(f\right)^{2}}\right)}$
and $\Delta f_{\theta}\geq1/2$. Since $\theta$ is the first time
that $f_{t}\geq0$, and since with probability $1$ there is no jump
at time $1$, this event contributes at least $1/4$ to $U_{\alpha}$;
thus 
\begin{align*}
\p\left[U_{\alpha}\geq\frac{1}{4}\right] & \geq\p\left[\left\{ \norm{\grad f_{\theta}}_{2}\leq\alpha\sqrt{\log\left(2+\frac{e}{\sum_{i}\Inf_{i}\left(f\right)^{2}}\right)}\right\} \intersect\left\{ \Delta f_{\theta}\geq\frac{1}{2}\right\} \right]\\
 & =\p\left[\norm{\grad f_{\theta}}_{2}\leq\alpha\sqrt{\log\left(2+\frac{e}{\sum_{i}\Inf_{i}\left(f\right)^{2}}\right)}\Bigr|\Delta f_{\theta}\geq\frac{1}{2}\right]\p\left[\Delta f_{\theta}\geq\frac{1}{2}\right]\\
\left(\ref{eq:case_large_jump}\right),\left(\ref{eq:prob_that_theta_occurs}\right) & \geq\p\left[\norm{\grad f_{\theta}}_{2}\leq\alpha\sqrt{\log\left(2+\frac{e}{\sum_{i}\Inf_{i}\left(f\right)^{2}}\right)}\Bigr|\Delta f_{\theta}\geq\frac{1}{2}\right]\frac{1}{8}\var\left(f\right).
\end{align*}
On the other hand, by Markov's inequality and Proposition \ref{prop:right_quadratic_variation_of_small_gradient_is_small},
this probability is upper-bounded by
\[
\p\left[U_{\alpha}\geq\frac{1}{4}\right]\leq\frac{\e\left[U_{\alpha}\right]}{1/4}\leq4\rho\left(\alpha\right)\var\left(f\right).
\]
Combining the two displays, we get 
\[
\p\left[\norm{\grad f_{\theta}}_{2}\leq\alpha\sqrt{\log\left(2+\frac{e}{\sum_{i}\Inf_{i}\left(f\right)^{2}}\right)}\Bigr|\Delta f_{\theta}\geq\frac{1}{2}\right]\leq4\rho\left(\alpha\right).
\]
Taking $\alpha$ small enough so that the right hand side is smaller
than $1/2$, we have 
\[
\p\left[H_{\alpha,\theta}\Bigr|\Delta f_{\theta}\geq\frac{1}{2}\right]>\frac{1}{2}.
\]
For this $\alpha$, under the event $H_{\alpha,\theta}\intersect\left\{ \theta<1\right\} $,
$\Psi_{\theta}$ is large:
\[
\Psi_{\theta}=\norm{\grad f_{\theta}}_{2}^{2p}\sup_{0\leq s\leq\theta}\frac{1+f_{s}}{2}\geq\norm{\grad f_{\theta}}_{2}^{2p}\frac{1+f_{\theta}}{2}\geq\frac{1}{2}\alpha\left(\log\left(2+\frac{e}{\sum_{i}\Inf_{i}\left(f\right)^{2}}\right)\right)^{p},
\]
and so 
\begin{align*}
\p\left[\tau_{\alpha}<1\Bigr|\theta<1\right] & \geq\p\left[\tau_{\alpha}<1\Bigr|\theta<1\text{ and }\Delta f_{\theta}\geq\frac{1}{2}\right]\p\left[\Delta f_{\theta}\geq\frac{1}{2}\Bigr|\theta<1\right]\\
\left(\ref{eq:case_large_jump}\right) & \geq\p\left[H_{\alpha,\theta}\Bigr|\Delta f_{\theta}\geq\frac{1}{2}\right]\cdot\frac{1}{2}\geq\frac{1}{4}
\end{align*}
as needed.

\end{proof}

\section{Talagrand's influence inequality and its stability}

The proofs of Theorems \ref{thm:talagrands_inequality} and \ref{thm:robust_implies_large_vertex_boundary}
are similar in spirit to that of Theorem \ref{thm:talagrands_conjecture},
and again require bounding the gain in quadratic variation. However,
extra care is needed to bound the size the individual influence processes
$f_{t}^{\left(i\right)}$. 

We first define several quantities which will be central to our proofs.
For a fixed $0<\alpha\leq1$ whose value will be chosen later, let
\begin{equation}
F_{\alpha}=\left\{ \exists t\in\left[0,1\right],\exists i\in\left[n\right]\mid t\in J_{i}\text{ and }f_{t}^{\left(i\right)}\geq\alpha\right\} \label{eq:definition_of_good_jumps}
\end{equation}
be the event that a coordinate had large derivative at the time it
jumped; let 
\[
Q_{\alpha}^{\left(i\right)}=2\int_{0}^{1}t\left(f_{t}^{\left(i\right)}\right)^{2}\one_{f_{t}^{\left(i\right)}<\alpha}dt,
\]
and 
\[
Q_{\alpha}=\sum_{i=1}^{n}Q_{\alpha}^{\left(i\right)};
\]
and let
\[
V_{\alpha}^{\left(i\right)}=\sum_{t\in J_{i}\intersect\left[0,1\right]}\left(2t\partial_{i}f_{t}\right)^{2}\one_{f_{t}^{\left(i\right)}<\alpha}
\]
and
\[
V_{\alpha}=\sum_{i=1}^{n}V_{\alpha}^{\left(i\right)}.
\]
$V_{\alpha}$ can be thought of as the quadratic variation of the
process $f_{t}$, but where big jumps (i.e those larger than $t\alpha$)
are excluded. Finally, define 
\[
\rho\left(x\right)=x\left(\log\frac{1}{x}+2\right).
\]
Instead of using Theorem \ref{thm:improved_keller_kindler} to bound
influences, we use the following lemma:
\begin{lem}
\label{lem:phi_can_be_bounded}There exists a universal constant $\gamma>1$
so that 
\begin{equation}
\varphi_{i}\left(s\right)\leq\gamma\Inf_{i}\left(f\right)^{1+s/\left(2\gamma\right)}\label{eq:phi_can_be_bounded}
\end{equation}
for all $0\leq s\leq\gamma$.
\end{lem}

This lemma can be derived from the hypercontractivity principle (see
e.g \cite{cordero_ledoux_hypercontractive_measures} and \cite[Cor. 9.25]{odonnell_analysis_of_boolean_functions}).
However, we give a different proof based on the analysis of the stochastic
process $f_{t}$; this analysis can be pushed further to obtain the
stability results. On an intuitive level and in light of equation
(\ref{eq:sum_of_jumps_to_integral_main}) the lemma shows that all
of the ``action'' which contributes to the variance of the function
happens very close to time $1$.

\begin{proof}[Proof of Lemma \ref{lem:phi_can_be_bounded}]
Let $\gamma>1$ to be chosen later. We start by showing that there
exists a constant $c_{\gamma}'>0$ such that 
\begin{equation}
\varphi_{i}\left(\gamma\right)\leq\gamma\varphi_{i}\left(0\right)^{1+c_{\gamma}'}.\label{eq:phi_at_const_is_small}
\end{equation}
Recall that $\psi_{i}\left(t\right)=\varphi_{i}\left(\log1/t\right)$;
by applying Corollary \ref{cor:differentiation_of_square} to the
function $f_{i}$, we see that $\psi_{i}$ satisfies 

\begin{equation}
\frac{d\psi_{i}}{dt}=2t\e\norm{\grad f_{t}^{\left(i\right)}}_{2}^{2}.\label{eq:original_ode_for_psi}
\end{equation}
The right hand side of equation (\ref{eq:original_ode_for_psi}) can
be bounded using Lemma \ref{lem:biased_level_1_inequality}: Taking
$g=f^{\left(i\right)}$ and $x=B_{t}$ in equation (\ref{eq:bound_on_gradient})
and substituting this in equation (\ref{eq:original_ode_for_psi}),
we have
\[
\frac{d\psi_{i}}{dt}\leq2t\frac{L}{\left(1-t\right)^{4}}\e\left[\left(f_{t}^{\left(i\right)}\right)^{2}\log\frac{e}{\left(f_{t}^{\left(i\right)}\right)^{2}}\right].
\]
For $t\leq1/2$, 
\begin{align}
\frac{d\psi_{i}}{dt} & \leq16L\e\left[\left(f_{t}^{\left(i\right)}\right)^{2}\log\frac{e}{\left(f_{t}^{\left(i\right)}\right)^{2}}\right]\nonumber \\
\left(\text{Jensen's inequality}\right) & \leq16L\e\left[\left(f_{t}^{\left(i\right)}\right)^{2}\right]\log\frac{e}{\e\left[\left(f_{t}^{\left(i\right)}\right)^{2}\right]}\nonumber \\
 & =16L\psi_{i}\left(t\right)\log\frac{e}{\psi_{i}\left(t\right)}.\label{eq:ode_upper_bound}
\end{align}
By Lemma \ref{lem:differential_inequality}, there exists a time $t_{0}\leq1/2$
and a constant $K$ such that for all $t\in\left[0,t_{0}\right]$,
\[
\psi_{i}\left(t\right)\leq K\psi_{i}\left(0\right)^{e^{-16Lt}}=K\Inf_{i}\left(f\right)^{2e^{-16Lt}},
\]
where in the last equality we used equation (\ref{eq:influence_by_harmonic_extension_of_absolute})
and the fact that $\psi_{i}\left(0\right)=\e\left(\left(f_{0}^{\left(i\right)}\right)^{2}\right)$.
Since $\varphi_{i}\left(s\right)=\psi_{i}\left(e^{-s}\right)$, if
$\gamma\geq\log\left(1/t_{0}\right)$ then 
\begin{align*}
\varphi_{i}\left(\gamma\right) & \leq K\Inf_{i}\left(f\right){}^{2e^{-16Le^{-\gamma}}}\\
 & =K\Inf_{i}\left(f\right)^{1+\left(2e^{e^{-16Le^{-\gamma}}}-1\right)}.
\end{align*}
Setting $c_{\gamma}'=2e^{-16Le^{-\gamma}}-1$ and taking $\gamma$
larger than $K$ gives the desired result: Equation (\ref{eq:phi_at_const_is_small})
follows because $\varphi_{i}\left(0\right)=\e\left(f_{1}^{\left(i\right)}\right)^{2}=\e\left(f^{\left(i\right)}\right)^{2}=\Inf_{i}\left(f\right)$
by equation (\ref{eq:influence_by_derivative}). Note that $c_{\gamma}'>0$
only if $\gamma>\log16L-\log\log2$.

Using equation (\ref{eq:phi_at_const_is_small}) together with the
log-convexity from Lemma \ref{lem:squared_functions_are_log_convex},
for all $0\leq s\leq\gamma$ we can bound $\varphi_{i}\left(s\right)$
by 
\begin{align*}
\varphi_{i}\left(s\right) & =\varphi_{i}\left(\left(1-\frac{s}{\gamma}\right)\cdot0+\frac{s}{\gamma}\cdot\gamma\right)\\
 & \leq\varphi_{i}\left(0\right)^{1-s/\gamma}\varphi_{i}\left(\gamma\right)^{s/\gamma}\\
 & \leq\varphi_{i}\left(0\right)^{1-s/\gamma}\varphi_{i}\left(0\right)^{\left(1+c_{\gamma}'\right)s/\gamma}\\
 & =\gamma\Inf_{i}\left(f\right){}^{1+c_{\gamma}s}
\end{align*}
as needed, with $c_{\gamma}=c_{\gamma}'/\gamma=\left(2e^{-16Le^{-\gamma}}-1\right)/\gamma$.
The theorem then follows by taking $\gamma$ large enough so that
$\gamma\geq\log16L-\log\log2$ and and $c_{\gamma}'\geq1/2$.
\end{proof}
The following two propositions are somewhat analogous to Propositions
\ref{prop:left_quadratic_variation_of_small_gradient_is_small} and
\ref{prop:quadratic_variation_is_not_concentrated_near_zero}.
\begin{prop}
\label{prop:Q_integral_is_smaller_than_phi_etc}Let $0<\alpha\leq1$.
Then 
\[
\e\left[Q_{\alpha}\right]\leq4\gamma^{2}\rho\left(\alpha\right)T\left(f\right),
\]
where $\gamma$ is the universal constant from Lemma \ref{lem:phi_can_be_bounded}.
\end{prop}

\begin{prop}
\label{prop:integral_cannot_be_too_small}If $0<\alpha<1/8$ and $\p\left[F_{\alpha}\right]<\frac{1}{200}\var\left(f\right)$,
then
\[
\p\left[V_{\alpha}\geq\frac{1}{64}\right]>\frac{1}{16}\var\left(f\right).
\]
\end{prop}

The proofs of Propositions \ref{prop:Q_integral_is_smaller_than_phi_etc}
and \ref{prop:integral_cannot_be_too_small} are of a very similar
nature to those of Theorem \ref{thm:improved_keller_kindler} and
Propositions \ref{prop:left_quadratic_variation_of_small_gradient_is_small}
and \ref{prop:quadratic_variation_is_not_concentrated_near_zero}.

\begin{proof}[Proof of Proposition \ref{prop:Q_integral_is_smaller_than_phi_etc}]
Since $Q_{\alpha}=\sum_{i=1}^{n}Q_{\alpha}^{\left(i\right)}$, it
is enough to show that 
\[
\e\left[Q_{\alpha}^{\left(i\right)}\right]\leq4\gamma^{2}\rho\left(\alpha\right)\frac{\Inf_{i}\left(f\right)}{1+\log\left(1/\Inf_{i}\left(f\right)\right)}.
\]
Denoting $\tilde{\varphi}_{i}\left(s\right)=\e\left(f_{e^{-s}}^{\left(i\right)}\right)^{2}\one_{f_{e^{-s}}^{\left(i\right)}<\alpha}$,
by change of variables we get
\begin{equation}
\e\left[Q_{\alpha}^{\left(i\right)}\right]\leq2\int_{0}^{\infty}e^{-2s}\tilde{\varphi}_{i}\left(s\right)ds.\label{eq:bounding_individual_influences_by_tilde_phi}
\end{equation}
Let $\tau=\frac{\log\left(1/\alpha\right)}{2+\frac{1}{2\gamma}\log\left(1/\Inf_{i}\left(f\right)\right)}$.
Assume first that $\tau\leq\frac{1}{2}\gamma$. The integral in equation
(\ref{eq:bounding_individual_influences_by_tilde_phi}) then splits
up into three parts:
\begin{equation}
\e\left[Q_{\alpha}^{\left(i\right)}\right]\leq2\int_{0}^{\tau}e^{-2s}\tilde{\varphi}_{i}\left(s\right)ds+2\int_{\tau}^{\gamma}e^{-2s}\tilde{\varphi}_{i}\left(s\right)ds+2\int_{\gamma}^{\infty}e^{-2s}\tilde{\varphi}_{i}\left(s\right)ds.\label{eq:splitting_the_integral}
\end{equation}
For the first integral on the right hand side, we write
\begin{align*}
\tilde{\varphi}_{i}\left(s\right) & \leq\alpha\e\abs{f_{e^{-s}}^{\left(i\right)}}\\
 & =\alpha\e f_{e^{-s}}^{\left(i\right)}=\alpha f_{0}^{\left(i\right)}=\alpha\Inf_{i}\left(f\right).
\end{align*}
Thus
\begin{align}
\int_{0}^{\tau}e^{-2s}\tilde{\varphi}_{i}\left(s\right)ds & \leq\alpha\tau\Inf_{i}\left(f\right)\nonumber \\
\left(\text{by choice of \ensuremath{\tau}}\right) & \leq2\gamma\frac{\Inf_{i}\left(f\right)}{1+\log\left(1/\Inf_{i}\left(f\right)\right)}\alpha\log\left(1/\alpha\right).\label{eq:bound_on_A}
\end{align}
For the second and third integrals, we use the fact that trivially,
$\tilde{\varphi}_{i}\left(s\right)\leq\varphi_{i}\left(s\right)$
for all $s$. By Lemma \ref{lem:phi_can_be_bounded}, for $s\in\left[\tau,\gamma\right]$
we then have $\tilde{\varphi}_{i}\left(s\right)\leq\gamma\Inf_{i}\left(f\right)^{1+s/2\gamma}$.
The second integral is therefore bounded by 
\begin{align}
\int_{\tau}^{\gamma}e^{-2s}\tilde{\varphi}_{i}\left(s\right)ds & \leq\gamma\int_{\tau}^{\gamma}e^{-2s}\Inf_{i}\left(f\right)^{1+s/2\gamma}ds\nonumber \\
 & \leq\gamma\int_{\tau}^{\infty}e^{-2s}\Inf_{i}\left(f\right)^{1+s/2\gamma}ds\nonumber \\
 & =\gamma\Inf_{i}\left(f\right)\int_{\tau}^{\infty}e^{s\left(\frac{1}{2\gamma}\log\Inf_{i}\left(f\right)-2\right)}ds\nonumber \\
 & \leq\gamma\frac{\Inf_{i}\left(f\right)}{2+\frac{1}{2\gamma}\log\left(1/\Inf_{i}\left(f\right)\right)}e^{\tau\left(\frac{1}{2\gamma}\log\Inf_{i}\left(f\right)-2\right)}\nonumber \\
 & \leq2\gamma^{2}\frac{\Inf_{i}\left(f\right)}{1+\log\left(1/\Inf_{i}\left(f\right)\right)}\alpha.\label{eq:bound_on_B}
\end{align}
For the third integral, we use the fact that $\varphi_{i}\left(s\right)$
is a decreasing function in $s$ (as can be seen from equation (\ref{eq:phi_as_sum_of_exponentials})).
Since $\gamma>1$ and $\tau\leq\frac{1}{2}\gamma$, we immediately
have $\int_{\gamma}^{\infty}e^{-2s}\tilde{\varphi}_{i}\left(s\right)ds\leq\int_{\tau}^{\gamma}e^{-2s}\tilde{\varphi}_{i}\left(s\right)ds$.
Putting these bounds together, when $\tau<\frac{1}{2}\gamma$ we get
that 
\begin{align*}
\e\left[Q_{\alpha}^{\left(i\right)}\right] & \leq2\left(2\gamma\frac{\Inf_{i}\left(f\right)}{1+\log\left(1/\Inf_{i}\left(f\right)\right)}\alpha\log\left(1/\alpha\right)+\left(2+2\right)\gamma^{2}\frac{\Inf_{i}\left(f\right)}{1+\log\left(1/\Inf_{i}\left(f\right)\right)}\alpha\right)\\
 & =4\gamma^{2}\rho\left(\alpha\right)\frac{\Inf_{i}\left(f\right)}{1+\log\left(1/\Inf_{i}\left(f\right)\right)}.
\end{align*}
Now assume that $\tau\geq\frac{1}{2}\gamma$. The integral in equation
(\ref{eq:bounding_individual_influences_by_tilde_phi}) then splits
up into two parts:
\[
\e\left[Q_{\alpha}^{\left(i\right)}\right]\leq2\int_{0}^{\tau}e^{-2s}\tilde{\varphi}_{i}\left(s\right)ds+2\int_{\tau}^{\infty}e^{-2s}\tilde{\varphi}_{i}\left(s\right)ds.
\]
Again, since $\varphi_{i}\left(s\right)$ is decreasing as a function
of $s$ and since $\tau\geq\frac{1}{2}\gamma>\frac{1}{2}$, the second
integral is smaller than the first, and so by (\ref{eq:bound_on_A}),
\begin{align*}
\e\left[Q_{\alpha}^{\left(i\right)}\right] & \leq2\cdot2\gamma^{2}\frac{\Inf_{i}\left(f\right)}{1+\log\left(1/\Inf_{i}\left(f\right)\right)}\alpha\log\left(1/\alpha\right)\leq4\gamma^{2}\rho\left(\alpha\right)\frac{\Inf_{i}\left(f\right)}{1+\log\left(1/\Inf_{i}\left(f\right)\right)}
\end{align*}
in this case as well.
\end{proof}
\begin{proof}[Proof of Proposition \ref{prop:integral_cannot_be_too_small}]
Assume without loss of generality that $f_{0}=\e f\leq0$ (if not,
use $-f$ instead of $f$; the variances and the probability $\p\left[V_{\alpha}\geq x\right]$
are the same for both functions). Let $\tau=\inf\left\{ 0\leq t\leq1\mid f_{t}\in\left(0,2\alpha\right)\right\} \land1$.
By conditioning on the event $\left\{ \tau<1\right\} $, for any $x>0$
we have 
\[
\p\left[\left[f\right]_{1}\geq x\right]\geq\p\left[\left[f\right]_{1}\geq x\Bigr|\tau<1\right]\p\left[\tau<1\right].
\]
We start by bounding the probability $\p\left[\tau<1\right]$. Let
$A=\left\{ \exists t\in\left[0,1\right]\text{ s.t }f_{t}>0\right\} $,
and observe that $\left\{ \tau<1\right\} \subseteq A$. Under the
event $A\backslash\left\{ \tau<1\right\} $, the process $f_{t}$
never visited the interval $\left(0,2\alpha\right)$ and yet at some
point reached a value larger than $0$, and so necessarily had a jump
discontinuity of size at least $2\alpha$. But a jump occurring at
time $t$ due to a discontinuity in $B_{t}^{\left(i\right)}$ is of
size $2t\abs{\partial_{i}f_{t}}$, and so $2t\abs{\partial_{i}f_{t}}\geq2\alpha$,
implying that $f_{t}^{\left(i\right)}\geq\abs{\partial_{i}f_{t}}\geq\alpha$.
Thus, $A\intersect\left\{ \tau=1\right\} \subseteq A\intersect F_{\alpha}$,
and so $A\intersect F_{\alpha}^{C}\subseteq A\intersect\left\{ \tau<1\right\} =\left\{ \tau<1\right\} $.
Hence 
\begin{align*}
\p\left[\tau<1\right] & \geq\p\left[A\backslash F_{\alpha}\right]\\
 & \geq\p\left[A\right]-\p\left[F_{\alpha}\right].
\end{align*}
To bound $\p\left[A\right]$, note that $\left\{ f_{1}=1\right\} \subseteq A$.
By the martingale property of $f_{t}$,
\[
f_{0}=\e f_{1}=2\p\left[f_{1}=1\right]-1,
\]
and so 
\begin{align*}
\p\left[A\right]\geq\p\left[f_{1}=1\right] & =\frac{1+f_{0}}{2}=\frac{1-f_{0}^{2}}{2\left(1-f_{0}\right)}=\frac{\var\left(f\right)}{2\left(1-f_{0}\right)}\geq\frac{1}{4}\var\left(f\right).
\end{align*}
Putting this together with the assumption that $\p\left[F_{\alpha}\right]<\frac{1}{200}\var\left(f\right)$
gives
\begin{equation}
\p\left[\tau<1\right]\geq\frac{1}{8}\var\left(f\right).\label{eq:probability_of_getting_far_when_starting_small}
\end{equation}
Next we bound the probability $\p\left[\left[f\right]_{1}\geq x\Bigr|\tau<1\right]$,
by relating the quadratic variation to the variance of $f_{t}$.

Let $\sigma$ be the stopping time $\sigma=\inf\left\{ s\geq\tau\mid\left[f\right]_{s}\geq x\right\} \land1$.
Since the $\sigma$-algebra generated by the event $\left\{ \tau<1\right\} $
is contained in that generated by $B_{\tau}$, 
\[
\e\left[\left(f_{\sigma}-f_{\tau}\right)^{2}\Bigr|\tau<1\right]=\e\left[\e\left[\left(f_{\sigma}-f_{\tau}\right)^{2}\Bigr|B_{\tau}\right]\mid\tau<1\right].
\]
Conditioned on $B_{\tau}$, the process $f_{t}$ is a martingale for
$t\in\left[\tau,\sigma\right]$, and so by (\ref{eq:variance_of_martingales_via_quadratic_variation}),
\begin{align}
\e\left[\e\left[\left(f_{\sigma}-f_{\tau}\right)^{2}\Bigr|B_{\tau}\right]\Bigr|\tau<1\right] & =\e\left[\e\left[\left(\left[f\right]_{\sigma}-\left[f\right]_{\tau}\right)\Bigr|B_{\tau}\right]\mid\tau<1\right]\nonumber \\
 & =\e\left[\left[f\right]_{\sigma}-\left[f\right]_{\tau}\Bigr|\tau<1\right]\nonumber \\
 & =\e\left[\left(\left[f\right]_{\sigma}-\left[f\right]_{\tau}\right)\one_{F_{\alpha}}\Bigr|\tau<1\right]+\e\left[\left(\left[f\right]_{\sigma}-\left[f\right]_{\tau}\right)\one_{F_{\alpha}^{C}}\Bigr|\tau<1\right].\label{eq:total_var_basic_tau_leq_1}
\end{align}
For the first term on the right hand side, observe that $\left[f\right]_{\sigma}-\left[f\right]_{\tau}\leq x+4$:
Since $\left[f\right]_{\sigma}$ is the sum of squares of the jumps
of $f_{t}$ up to time $\sigma$, the largest value it can attain
is $x$ plus the square of the jump which occurred at time $\sigma$,
and the size of this jump is bounded by $2$. Thus 
\begin{align}
\e\left[\left(\left[f\right]_{\sigma}-\left[f\right]_{\tau}\right)\one_{F_{\alpha}}\Bigr|\tau<1\right] & \leq\left(x+4\right)\e\left[\one_{F_{\alpha}}\Bigr|\tau<1\right]\nonumber \\
 & =\left(x+4\right)\frac{\p\left[F_{\alpha}\intersect\left\{ \tau<1\right\} \right]}{\p\left[\tau<1\right]}\nonumber \\
 & \leq\left(x+4\right)\frac{\p\left[F_{\alpha}\right]}{\p\left[\tau<1\right]}\leq\frac{8\left(x+4\right)}{200},\label{eq:total_var_when_F}
\end{align}
where the last inequality is by the assumption on $\p\left[F_{\alpha}\right]$
and equation (\ref{eq:probability_of_getting_far_when_starting_small}). 

For the second term on the right hand side, since the event $\one_{F_{\alpha}^{C}}$
forces all jumps to be of size smaller than $2\alpha$, we similarly
have 
\begin{equation}
\left(\left[f\right]_{\sigma}-\left[f\right]_{\tau}\right)\one_{F_{\alpha}^{C}}\leq x+4\alpha^{2}.\label{eq:total_var_when_F_complement}
\end{equation}
Plugging displays (\ref{eq:total_var_when_F}) and (\ref{eq:total_var_when_F_complement})
into (\ref{eq:total_var_basic_tau_leq_1}), we get
\begin{equation}
\e\left[\left(f_{\sigma}-f_{\tau}\right)^{2}\Bigr|\tau<1\right]\leq x+4\alpha^{2}+\frac{8\left(x+4\right)}{200}.\label{eq:bounding_when_tau_is_1}
\end{equation}
On the other hand, 
\begin{align*}
\e\left[\left(f_{\sigma}-f_{\tau}\right)^{2}\Bigr|\tau<1\right] & \geq\e\left[\left(f_{\sigma}-f_{\tau}\right)^{2}\Bigr|\tau<1\text{ and }\abs{f_{\sigma}}=1\right]\p\left[\abs{f_{\sigma}}=1\Bigr|\tau<1\right]\\
 & \geq\left(1-2\alpha\right)^{2}\p\left[\abs{f_{\sigma}}=1\Bigr|\tau<1\right],
\end{align*}
and so together with (\ref{eq:bounding_when_tau_is_1}) and plugging
in $x=1/64$ and $\alpha<1/8$,
\[
\p\left[\abs{f_{\sigma}}<1\Bigr|\tau<1\right]\geq1-\frac{x+4\alpha^{2}+\frac{8\left(x+4\right)}{200}}{\left(1-2\alpha\right)^{2}}\geq\frac{259}{450}>\frac{5}{9}.
\]
Now, if $\abs{f_{\sigma}}\neq1$ then $f_{\sigma}$ stopped because
$\left[f\right]_{\sigma}$ was larger than or equal to $x$. Since
$\left[f\right]_{s}$ is increasing as a function of $s$, $\left[f\right]_{1}\geq x$
as well, and so 
\begin{equation}
\p\left[\left[f\right]_{1}\geq\frac{1}{64}\Bigr|\tau<1\right]\geq\p\left[\abs{f_{\sigma}}<1\Bigr|\tau<1\right]\geq\frac{5}{9}.\label{eq:V_tilde_is_large_given_tau_smaller_than_1}
\end{equation}
Combining (\ref{eq:probability_of_getting_far_when_starting_small})
and (\ref{eq:V_tilde_is_large_given_tau_smaller_than_1}) gives
\[
\p\left[\left[f\right]_{1}\geq\frac{1}{64}\right]\geq\frac{5}{72}\var\left(f\right).
\]
Under the event $F_{\alpha}^{C}$ we have that $V_{\alpha}=\left[f\right]_{1}$,
and by a union bound we get %

\begin{align*}
\p\left[V_{\alpha}\geq\frac{1}{64}\right] & \geq\p\left[\left[f\right]_{1}\geq\frac{1}{64}\right]-\p\left[F_{\alpha}\right]\\
 & \geq\frac{5}{72}\var\left(f\right)-\frac{1}{200}\var\left(f\right)>\frac{1}{16}\var\left(f\right).
\end{align*}
\end{proof}

\subsection{Proof of Theorem \ref{thm:talagrands_inequality}}
\begin{proof}[Proof of Theorem \ref{thm:talagrands_inequality}]
Let $\gamma$ be the constant from the statement of Lemma \ref{lem:phi_can_be_bounded}.
By Proposition \ref{prop:Q_integral_is_smaller_than_phi_etc}, for
every $0<\alpha\leq1$, we have 
\[
\e\left[Q_{\alpha}\right]\leq4\gamma^{2}\rho\left(\alpha\right)T\left(f\right).
\]
Choosing $\alpha=1$ just gives $Q_{\alpha}=2\sum_{i=1}^{n}\int_{0}^{1}t\left(f_{t}^{\left(i\right)}\right)^{2}dt$,
since the derivatives are bounded by $1$; the expectation of this
expression, as seen in (\ref{eq:var_is_bounded_by_influence_process}),
is larger than $\var\left(f\right)$. We thus have
\[
\var\left(f\right)\leq\e\left[Q_{\alpha}\right]\leq4\gamma^{2}\rho\left(1\right)T\left(f\right)=8\gamma^{2}\cdot T\left(f\right).
\]
 
\end{proof}

\subsection{Proof of Theorem \ref{thm:robust_implies_large_vertex_boundary}}

Using Propositions \ref{prop:Q_integral_is_smaller_than_phi_etc}
and \ref{prop:integral_cannot_be_too_small}, we can obtain the following
lemma. 
\begin{lem}
\label{lem:bounding_the_probability_of_F}Let $\gamma$ be the constant
from Lemma \ref{lem:phi_can_be_bounded}, and assume that $0<\alpha<1/16$
is small enough so that $4\gamma^{2}\rho\left(\alpha\right)T\left(f\right)\leq\frac{1}{1024}\var\left(f\right)$.
Then 
\[
\p\left[F_{\alpha}\right]\geq\frac{1}{200}\var\left(f\right).
\]
\end{lem}

\begin{proof}
Suppose by contradiction that $\p\left[F_{\alpha}\right]<\frac{1}{200}\var\left(f\right)$.
Since neither $\partial_{i}f_{t}$ nor $f_{t}^{\left(i\right)}$ depend
on coordinate $i$, we may write 
\[
\e\left[V_{\alpha}^{\left(i\right)}\right]=\e\sum_{t\in J_{i}\intersect\left[0,1\right]}\left(2t\partial_{i}f_{t}\right)^{2}\one_{f_{t}^{\left(i\right)}<\alpha}=\e\sum_{t\in J_{i}\intersect\left[0,1\right]}\left(2t\partial_{i}f_{t^{-}}\right)^{2}\one_{f_{t^{-}}^{\left(i\right)}<\alpha}.
\]
Invoking Lemma \ref{lem:sum_of_jumps_to_integral} with $g_{t}=\left(\partial_{i}f_{t^{-}}\right)^{2}\one_{f_{t^{-}}^{\left(i\right)}<\alpha}$,
we get 
\begin{align*}
\e\left[V_{\alpha}^{\left(i\right)}\right] & =\e\sum_{t\in J_{i}\intersect\left[0,1\right]}\left(2t\partial_{i}f_{t^{-}}\right)^{2}\one_{f_{t^{-}}^{\left(i\right)}<\alpha}\\
 & \leq2\e\int_{0}^{1}t\left(\partial_{i}f_{t^{-}}\right)^{2}\one_{f_{t^{-}}^{\left(i\right)}<\alpha}dt\\
 & =2\e\int_{0}^{1}t\left(\partial_{i}f_{t}\right)^{2}\one_{f_{t}^{\left(i\right)}<\alpha}dt=\e\left[Q_{\alpha}^{\left(i\right)}\right].
\end{align*}
Using Proposition \ref{prop:Q_integral_is_smaller_than_phi_etc},
this means that 
\[
\e\left[V_{\alpha}\right]\leq4\gamma^{2}\rho\left(\alpha\right)T\left(f\right).
\]
On the other hand, by Proposition \ref{prop:integral_cannot_be_too_small}
and Markov's inequality, 
\begin{align*}
\e\left[V_{\alpha}\right] & \geq\p\left[V_{\alpha}\geq\frac{1}{64}\right]\cdot\frac{1}{64}>\frac{1}{1024}\var\left(f\right),
\end{align*}
contradicting the assumption that $4\gamma^{2}\rho\left(\alpha\right)T\left(f\right)\leq\frac{1}{1024}\var\left(f\right)$.
\end{proof}
The main assertion involved in proving Theorem \ref{thm:robust_implies_large_vertex_boundary}
connects between the vertex boundary and the probability that the
function makes a large jump. 
\begin{prop}
\label{prop:using_the_event_E}For $0<\alpha\leq1$, let $F_{\alpha}$
be the event defined in equation (\ref{eq:definition_of_good_jumps}).
Then 
\begin{equation}
\mu\left(\partial^{\pm}f\right)\geq\frac{1}{2}\alpha\p\left[F_{\alpha}\right].\label{eq:using_event_E_display}
\end{equation}
\end{prop}

The prove this proposition, we will construct a modification $\tilde{B}_{t}$
of $B_{t}$, which can be thought of as a ``hesitant'' version of
$B_{t}$. For each coordinate $i$, let $\tilde{J}_{i}$ be the jump
set of a Poisson point process on $\left(0,1\right]$ with intensity
$1/2t$, independent from $B_{t}$ (and in particular, independent
from the jump process $J_{i}=\mathrm{Jump}\left(B_{t}^{\left(i\right)}\right)$).
Define $\tilde{B}_{t}=\left(\tilde{B}_{t}^{\left(1\right)},\ldots,\tilde{B}_{t}^{\left(n\right)}\right)$
to be the process such that for every $i$, 
\[
\tilde{B}_{t}^{\left(i\right)}=\begin{cases}
0 & t\in J_{i}\union\tilde{J}_{i}\\
B_{t}^{\left(i\right)} & \text{o.w}.
\end{cases}
\]
Loosely speaking, there are several ways of thinking about $\tilde{B}_{t}^{\left(i\right)}$:
\begin{enumerate}
\item The process $\tilde{B}_{t}^{\left(i\right)}$ can be seen as a ``hesitant''
variation of $B_{t}^{\left(i\right)}$: It jumps with twice the rate
(since its set of discontinuities is the union of two Poisson processes
with rate $1/2t$), but half of those times, it returns to the original
sign rather inverting it. We refer to this view as the ``standard
coupling'' of $\tilde{B}_{t}^{\left(i\right)}$ with $B_{t}^{\left(i\right)}$:
The process $\tilde{B}_{t}^{\left(i\right)}$ is a copy of $B_{t}^{\left(i\right)}$,
but with additional independent hesitant jumps. 
\item The process $\tilde{B}_{t}^{\left(i\right)}$ is equal to $0$ at
a discrete set of times which follows the law of a Poisson point process
with intensity $1/t$ (this is the union $J_{i}\union\tilde{J}_{i}$);
between two successive zeros it chooses randomly to be either $t$
or $-t$, each with probability $1/2$. 
\end{enumerate}
Similarly to the notation using $B_{t}$, we write $\tilde{f}_{t}=f\left(\tilde{B}_{t}\right)$,
and analogously $\partial_{i}\tilde{f}_{t}$, $\grad\tilde{f}_{t}$
and $\tilde{f}_{t}^{\left(i\right)}$.
\begin{lem}
The process $\tilde{B}_{t}$ is a martingale.
\end{lem}

\begin{proof}
Let $0\leq s<t\leq1$. For $s=0$, since $\tilde{B}_{0}^{\left(i\right)}=0$
always, we trivially have $\e\left[\tilde{B}_{t}^{\left(i\right)}\right]=0$,
so assume $s>0$. 

If $\tilde{B}_{s}^{\left(i\right)}=0$, then $s\in J_{i}\union\tilde{J}_{i}$.
Being independent Poisson point processes, almost surely we have $J_{i}\intersect\tilde{J}_{i}=\emptyset$,
and $\p\left[s\in J_{i}\Bigr|\tilde{B}_{s}^{\left(i\right)}\right]=\p\left[s\in\tilde{J}_{i}\Bigr|\tilde{B}_{s}^{\left(i\right)}\right]=1/2$.
Since $\tilde{B}_{t}^{\left(i\right)}=B_{t}^{\left(i\right)}$ almost
surely, we thus have 
\begin{align*}
\e\left[\tilde{B}_{t}^{\left(i\right)}\Bigr|\tilde{B}_{s}^{\left(i\right)}\right] & =\frac{1}{2}\e\left[B_{t}^{\left(i\right)}\Bigr|s\in J_{i},\tilde{B}_{s}^{\left(i\right)}\right]+\frac{1}{2}\e\left[B_{t}^{\left(i\right)}\Bigr|s\in\tilde{J}_{i},\tilde{B}_{s}^{\left(i\right)}\right]\\
 & =\frac{1}{2}\e\left[B_{t}^{\left(i\right)}\Bigr|s\in J_{i},B_{s}^{\left(i\right)}\right]+\frac{1}{2}\e\left[B_{t}^{\left(i\right)}\Bigr|s\in\tilde{J}_{i},B_{s}^{\left(i\right)}\right].
\end{align*}
It is evident by the definition of the process $B_{t}$ that
\[
\e\left[B_{t}^{\left(i\right)}\Bigr|s\in J_{i},B_{s}^{\left(i\right)}\right]=-\e\left[B_{t}^{\left(i\right)}\Bigr|s\in\tilde{J}_{i},B_{s}^{\left(i\right)}\right],
\]
so that $\e\left[\tilde{B}_{t}^{\left(i\right)}\Bigr|\tilde{B}_{s}^{\left(i\right)}\right]=0=\tilde{B}_{s}^{\left(i\right)}$. 

Finally, if $\tilde{B}_{s}^{\left(i\right)}\neq0$, then since $\tilde{B}_{t}^{\left(i\right)}\neq0$
almost surely, we have by (\ref{eq:sign_jumps_probabilities}) that
\[
\p\left[\sign\tilde{B}_{t}^{\left(i\right)}\neq\sign\tilde{B}_{s}^{\left(i\right)}\Bigr|\tilde{B}_{s}^{\left(i\right)}\right]=\frac{t-s}{2t}.
\]
Thus
\begin{align*}
\e\left[\tilde{B}_{t}^{\left(i\right)}\Bigr|\tilde{B}_{s}^{\left(i\right)}\right] & =\sign\tilde{B}_{s}^{\left(i\right)}\cdot t\cdot\frac{t-s}{2t}+\sign\tilde{B}_{s}^{\left(i\right)}\cdot\left(-t\right)\cdot\frac{t-s}{2t}\\
 & =\sign\tilde{B}_{s}^{\left(i\right)}\cdot s\\
 & =\tilde{B}_{s}^{\left(i\right)}.
\end{align*}
\end{proof}
\begin{proof}[Proof of Proposition \ref{prop:using_the_event_E}]
In order to distinguish between the vertex boundaries, we will use
the hesitant jump process $\tilde{f}_{t}$ defined above. We prove
(\ref{eq:using_event_E_display}) for the inner vertex boundary $\partial^{+}$;
the proof for $\partial^{-}$ is identical. Let $\tau=\inf\left\{ t>0\mid\exists i\in\left[n\right]\text{ s.t }\tilde{B}_{t}^{\left(i\right)}=0\text{ and }f_{t}^{\left(i\right)}\geq\alpha\right\} \land1$.
Note that for any $t_{0}>0$, we almost surely have that $\tilde{B}_{t}^{\left(i\right)}=0$
only finitely many times for $t\in\left[t_{0},1\right]$. Thus, if
$0<\tau<1$, then the infimum in the definition of $\tau$ is attained
as a minimum, and there exists an $i_{0}$ such that $\tilde{B}_{\tau}^{\left(i_{0}\right)}=0$
and $f_{\tau}^{\left(i_{0}\right)}\geq\alpha$. In fact, this holds
true if $\tau=0$ as well: In this case, there is a sequence of times
$t_{k}\to0$ and indices $i_{k}$ such that $f_{t_{k}}^{\left(i_{k}\right)}\geq\alpha$
and $\tilde{B}_{t_{k}}^{\left(i_{k}\right)}=0$. Since there are only
finitely many indices, there is a subsequence $k_{\ell}$ so that
$i_{k_{\ell}}$ are all the same index $i_{0}$, and the claim follows
by continuity of $f^{\left(i_{0}\right)}$ and the fact that $\tilde{B}_{0}=0$.

When $F_{\alpha}$ occurs, we necessarily have $\tau<1$, since $\tilde{B}_{t}^{\left(i\right)}=0$
whenever $B_{t}^{\left(i\right)}$ is discontinuous. Since $\tilde{B}_{1}$
is uniform on the hypercube, 
\begin{align*}
\mu\left(\partial^{+}f\right)=\p\left[\tilde{B}_{1}\in\partial^{+}f\right] & \geq\p\left[\tilde{B}_{1}\in\partial^{+}f\Bigr|\tau<1\right]\p\left[\tau<1\right]\\
 & \geq\p\left[\tilde{B}_{1}\in\partial^{+}f\Bigr|\tau<1\right]\p\left[F_{\alpha}\right],
\end{align*}
and so it suffices to show that

\begin{equation}
\p\left[\tilde{B}_{1}\in\partial^{+}f\Bigr|\tau<1\right]\geq\frac{1}{2}\alpha.\label{eq:modified_B_gives_large_boundary}
\end{equation}
Supposing that $\tau<1$, denote by $i_{0}$ a coordinate for which
$\tilde{B}_{\tau}^{\left(i_{0}\right)}=0$ and $f_{\tau}^{\left(i_{0}\right)}\geq\alpha$.
If $\tau=0$ then $\tilde{B}_{\tau}=B_{\tau}$; otherwise, almost
surely $i_{0}$ is the only coordinate of $\tilde{B}_{\tau}$ which
is $0$, and so $\tilde{B}_{\tau}^{\left(j\right)}=B_{\tau}^{\left(j\right)}$
for all $j\neq i_{0}$ almost surely. Since the function $f^{\left(i\right)}\left(\cdot\right)$
does not depend on the $i$-th coordinate, we deduce that $\tilde{f}_{\tau}^{\left(i_{0}\right)}=f_{\tau}^{\left(i_{0}\right)}\geq\alpha$
almost surely. Thus, under $\tau<1$, by the martingale property of
$\tilde{f}_{t}^{\left(i_{0}\right)}$, we have 
\begin{equation}
\p\left[\tilde{f}_{1}^{\left(i_{0}\right)}=1\Bigr|\tilde{B}_{\tau}\right]\geq\alpha.\label{eq:coupling_derivatives}
\end{equation}
Similarly, using the martingale property of $\tilde{B}_{t}^{\left(i_{0}\right)}$,
we have $\e\left[\tilde{B}_{1}^{\left(i_{0}\right)}\mid\tilde{B}_{\tau}\right]=\tilde{B}_{\tau}^{\left(i_{0}\right)}=0$,
and so 
\[
\p\left[\tilde{B}_{1}^{\left(i_{0}\right)}=1\Bigr|\tilde{B}_{\tau}\right]=\p\left[\tilde{B}_{1}^{\left(i_{0}\right)}=-1\Bigr|\tilde{B}_{\tau}\right]=\frac{1}{2}.
\]
Since $\partial_{i_{0}}\tilde{f}_{t}$ is independent of $\tilde{B}_{t}^{\left(i_{0}\right)}$,
we finally obtain 
\begin{align*}
\p\left[\tilde{B}_{1}\in\partial^{+}f\Bigr|\tilde{B}_{\tau}\right] & =\p\left[\tilde{B}_{1}^{\left(i_{0}\right)}=1\land\partial_{i_{0}}\tilde{f}_{1}=1\Bigr|\tilde{B}_{\tau}\right]+\p\left[\tilde{B}_{1}^{\left(i_{0}\right)}=-1\land\partial_{i_{0}}\tilde{f}_{1}=-1\Bigr|\tilde{B}_{\tau}\right]\\
\left(\text{independence}\right) & =\frac{1}{2}\p\left[\partial_{i_{0}}\tilde{f}_{1}=1\Bigr|\tilde{B}_{\tau}\right]+\frac{1}{2}\p\left[\partial_{i_{0}}\tilde{f}_{1}=-1\Bigr|\tilde{B}_{\tau}\right]\\
 & =\frac{1}{2}\p\left[\tilde{f}_{1}^{\left(i_{0}\right)}=1\Bigr|\tilde{B}_{\tau}\right]\\
 & \stackrel{\left(\ref{eq:coupling_derivatives}\right)}{\geq}\frac{1}{2}\alpha.
\end{align*}
\end{proof}
\begin{proof}[Proof of Theorem \ref{thm:robust_implies_large_vertex_boundary}]
Let $\gamma$ be the constant from the statement of Lemma \ref{lem:phi_can_be_bounded}.
Let $\alpha$ be such that $\alpha\log\frac{1}{\alpha}=\frac{1}{2^{14}\gamma^{2}}r_{\mathrm{Tal}}$.
Then the condition $4\gamma^{2}\rho\left(\alpha\right)T\left(f\right)\leq\frac{1}{1024}\var\left(f\right)$
is satisfied in Lemma \ref{lem:bounding_the_probability_of_F}, implying
that $\p\left[F_{\alpha}\right]\geq\frac{1}{200}\var\left(f\right)$.
Together with Proposition \ref{prop:using_the_event_E}, we have 
\begin{equation}
\mu\left(\partial^{\pm}f\right)\geq\frac{1}{2}\alpha\p\left[F_{\alpha}\right]\geq\frac{1}{400}\alpha\var\left(f\right).\label{eq:one_last_step}
\end{equation}
All that remains is to obtain a lower bound on $\alpha$. To this
end, observe that $\alpha\log\frac{1}{\alpha}\leq\sqrt{\alpha}$ for
all $\alpha\in\left[0,1\right]$, and so $\alpha\geq\frac{1}{2^{28}\gamma^{4}}r_{\mathrm{Tal}}^{2}$.
Thus $\log\frac{1}{\alpha}\leq\log\left(\frac{2^{28}\gamma^{4}}{r_{\mathrm{Tal}}^{2}}\right)$,
so there exists a constant $C_{B}'$ such that 
\[
\alpha=\frac{r_{\mathrm{Tal}}}{2^{14}\gamma^{2}\log\frac{1}{\alpha}}\geq\frac{r_{\mathrm{Tal}}}{C_{B}'\log\frac{C_{B}'}{r_{\mathrm{Tal}}}}.
\]
Plugging this into (\ref{eq:one_last_step}) gives the desired result.
\end{proof}

\section{Proof of Theorem \ref{thm:improved_keller_kindler}\label{sec:keller_kindler_proof}}

As explained above in equation (\ref{eq:keller_kindler_in_nice_form}),
our goal is to show that 
\begin{equation}
S_{\eps}\left(f\right)=\e\sum_{i=1}^{n}\int_{0}^{\sqrt{1-\eps}}t\left(\partial_{i}f_{t}\right)^{2}dt\leq C\var\left(f\right)\left(\sum_{i=1}^{n}\Inf_{i}\left(f\right)^{2}\right)^{c\eps}.\label{eq:again_keller_kindler_nice}
\end{equation}
We first show that we may assume that $f$ is monotone. For an index
$i=1,\ldots,n$, define an operator $\kappa_{i}$ by 
\[
\left(\kappa_{i}f\right)\left(y\right)=\begin{cases}
\max\left\{ f\left(y\right),f\left(y^{\oplus i}\right)\right\}  & y_{i}=1,\\
\min\left\{ f\left(y\right),f\left(y^{\oplus i}\right)\right\}  & y_{i}=0.
\end{cases}
\]
The following lemma relates between the influences and sensitivities
of $\kappa_{i}f$ and $f$:
\begin{lem}[{\cite[Lemma 2.7]{benjamini_kalai_schramm_noise_sensitivity}}]
$\kappa_{1}\kappa_{2}\ldots\kappa_{n}f$ is monotone, and for every
pair of indices $i,j$, $\Inf_{i}\left(\kappa_{j}f\right)\leq\Inf_{i}\left(f\right)$
and $S_{\eps}\left(\kappa_{i}f\right)\geq S_{\eps}\left(f\right)$.
\end{lem}

Thus, if equation (\ref{eq:again_keller_kindler_nice}) holds for
$\tilde{f}=\kappa_{1}\ldots\kappa_{n}f$, then it holds for $f$ as
well, since $S_{\eps}\left(f\right)\leq S_{\eps}\left(\tilde{f}\right)$
and $\sum_{i=1}^{n}\Inf_{i}\left(f\right)^{2}\geq\sum_{i=1}^{n}\Inf_{i}\left(\tilde{f}\right)^{2}$.
So it's enough to verify (\ref{eq:again_keller_kindler_nice}) for
monotone functions.

In order to prove (\ref{eq:again_keller_kindler_nice}), we may also
assume that for any fixed universal constant $K<1$, 
\begin{equation}
\sum_{i=1}^{n}\Inf_{i}\left(f\right)^{2}\leq K.\label{eq:can_assume_R0_is_small}
\end{equation}
For if $\sum_{i=1}^{n}\Inf_{i}\left(f\right)^{2}\geq K$ for some
$K$, then since $f$ is monotone, 
\[
\var\left(f\right)=\sum_{S\subseteq\left[n\right],S\neq\emptyset}\hat{f}\left(S\right)^{2}\geq\sum_{i=1}^{n}\hat{f}\left(\left\{ i\right\} \right)^{2}\stackrel{\left(\ref{eq:influence_of_monotone_function}\right)}{=}\sum_{i=1}^{n}\Inf_{i}\left(f\right)^{2}\geq K,
\]
and so $\var f\cdot\left(\sum_{i}\Inf_{i}\left(f\right)^{2}\right)\geq K^{2}$.
Equation (\ref{eq:again_keller_kindler_nice}) then holds trivially
with $C=1/K^{2}$ and $c=1$, since $S_{\eps}\leq1$ for all $\eps$.

Similarly, we may assume that for any fixed, universal constant $K$,
\begin{equation}
\var\left(f\right)\leq K;\label{eq:can_assume_variance_is_small}
\end{equation}
otherwise Theorem \ref{thm:improved_keller_kindler} would be equivalent
(up to constants) to the original theorem proved in \cite{keller_kindler_quantitative_noise_sensitivity}.
\begin{rem}
Our proof actually recovers the original theorem proved in \cite{keller_kindler_quantitative_noise_sensitivity},
but we make this assumption since it simplifies some bounds.
\end{rem}

Define $R\left(t\right)=\e\sum_{i}\left(\partial_{i}f_{t}\right)^{2}=\e\norm{\grad f_{t}}_{2}^{2}$.
At time $0$, we have 
\[
R\left(0\right)=\sum_{i=1}^{n}\left(\partial_{i}f_{0}\right)^{2}=\sum_{i=1}^{n}\left(\hat{\partial_{i}f}\left(\emptyset\right)\right)^{2}=\sum_{i=1}^{n}\hat{f}\left(\left\{ x_{i}\right\} \right)^{2}=\sum_{i=1}^{n}\Inf_{i}\left(f\right)^{2}.
\]
The function $R\left(t\right)$ is monotone in $t$: Since $\partial_{i}f_{t}$
is a martingale, $\left(\partial_{i}f_{t}\right)^{2}$ is a submartingale
and so $\e\left(\partial_{i}f_{t}\right)^{2}$ is increasing. 

By invoking Corollary \ref{cor:differentiation_of_square} on $\partial_{i}f$,
for every index $i$ we have 
\[
\frac{d}{dt}\e\left(\partial_{i}f_{t}\right)^{2}=2t\e\sum_{j=1}^{n}\left(\partial_{j}\partial_{i}f_{t}\right)^{2}.
\]
Thus 
\begin{equation}
\frac{d}{dt}R\left(t\right)=2t\e\sum_{i=1}^{n}\sum_{j=1}^{n}\left(\partial_{i}\partial_{j}f_{t}\right)^{2}\leq2\e\norm{\grad^{2}f_{t}}_{HS}^{2},\label{eq:derivative_of_q}
\end{equation}
where $\norm X_{HS}$ is the Hilbert-Schmidt norm of a matrix. By
Lemma \ref{lem:biased_level_2_inequality}, there exists a continuous
positive function $C\left(t\right)$ such that 
\begin{align*}
\frac{d}{dt}R\left(t\right) & \leq2C\left(t\right)\e\left[\norm{\grad f_{t}}_{2}^{2}\log\frac{C\left(t\right)}{\norm{\grad f_{t}}_{2}^{2}}\right]\\
\left(\text{Jensen's inequality}\right) & \leq2C\left(t\right)\e\left[\norm{\grad f_{t}}_{2}^{2}\right]\log\frac{C\left(t\right)}{\e\left[\norm{\grad f_{t}}_{2}^{2}\right]}=2C\left(t\right)R\left(t\right)\log\left(\frac{C\left(t\right)}{R\left(t\right)}\right).
\end{align*}
Since $C\left(t\right)$ is continuous it is bounded in $\left[0,1/2\right]$,
so there exists a constant $c>0$ such that for all $t\in\left[0,1/2\right]$,
\begin{equation}
\frac{d}{dt}R\left(t\right)\leq c\cdot R\left(t\right)\log\left(\frac{c}{R\left(t\right)}\right).\label{eq:derivative_of_r_is_bounded}
\end{equation}
By (\ref{eq:can_assume_R0_is_small}), we can assume that $R\left(0\right)\leq c/2$.
Using (\ref{eq:derivative_of_r_is_bounded}) together with Lemma \ref{lem:differential_inequality},
there exist constants $C,L>0$ and a time $t_{0}$, all of which depend
only on $c$, such that for all $t\in\left[0,t_{0}\right]$,

\[
R\left(t\right)\leq L\cdot R\left(0\right)^{e^{-Ct}}.
\]
In particular, there exists a constant $K>0$ such that 
\begin{equation}
R\left(e^{-K}\right)\leq L\cdot R\left(0\right)^{5/6},\label{eq:bound_on_R_at_constant_place}
\end{equation}
and since $R$ is increasing, we can always assume that $K>1$. Denote
$G\left(s\right)=R\left(e^{-s}\right)$; by Lemma \ref{lem:squared_functions_are_log_convex},
$G\left(s\right)$ is log-convex in $s$.
\begin{lem}
\label{lem:inequality_for_log_convex}Let $K\geq1$ and let $G\left(s\right)$
be a log-convex decreasing function. Denote $v=\int_{0}^{K}e^{-2s}G\left(s\right)ds$
and assume that $v>G\left(K\right)$. Then for all $r<K$, 
\[
\int_{0}^{r}e^{-2s}G\left(s\right)ds\geq v\cdot\left(1-\left(\frac{G\left(K\right)}{v}\right)^{r/K}\right).
\]
\end{lem}

\begin{proof}
Consider the function 
\[
h_{\ell}\left(s\right)=\frac{v\ell}{1-e^{-\ell K}}e^{-\left(\ell-2\right)s},
\]
where $\ell$ is the largest solution to the equation $h_{\ell}\left(K\right)=G\left(K\right)$.
By choice of $h_{\ell}$, we have 
\[
\int_{0}^{K}e^{-2s}h_{\ell}\left(s\right)ds=v=\int_{0}^{K}e^{-2s}G\left(s\right)ds.
\]
Since $e^{-2s}h_{\ell}\left(s\right)$ is log-linear on $\left[0,K\right]$,
$e^{-2s}G\left(s\right)$ is log-convex on $\left[0,K\right]$, they
have the same integral on $\left[0,K\right]$ and $G\left(K\right)=h_{\ell}\left(K\right)$,
we must have one of two cases:
\begin{enumerate}
\item $h_{\ell}\left(s\right)=G\left(s\right)$
\item The functions intersect at most once in the interval $\left[0,K\right)$
at some point $s_{0}$ such that $G\left(s\right)\geq h_{\ell}\left(s\right)$
for all $s<s_{0}$. 
\end{enumerate}
In either case, for all $r\in\left[0,K\right]$, we have
\begin{equation}
\int_{0}^{r}e^{-2s}G\left(s\right)ds\geq\int_{0}^{r}e^{-2s}h_{\ell}\left(s\right)ds=\frac{v}{1-e^{-K\ell}}\left(1-e^{-r\ell}\right)\geq v\left(1-e^{-r\ell}\right).\label{eq:technical_middle_of_keller_kindler}
\end{equation}
On the other hand, we chose $\ell$ to be such that $h_{\ell}\left(K\right)=G\left(K\right)$,
and so
\[
\frac{\ell}{1-e^{-\ell K}}e^{-\left(\ell-2\right)K}=\frac{h_{\ell}\left(K\right)}{v}=\frac{G\left(K\right)}{v}<1,
\]
where the last inequality is by assumption on $v$. The function $\frac{x}{1-e^{-xK}}e^{-\left(x-2\right)K}$
is decreasing as a function of $x$ in the interval $\left[2,\infty\right)$,
but is greater than 1 at $x=2$; hence, since $\ell$ is the largest
number for which $h_{\ell}\left(K\right)=G\left(K\right)$, we must
have $\ell>2$. We then have
\[
\ell e^{-\left(\ell-2\right)K}=\frac{G\left(K\right)\left(1-e^{-\ell K}\right)}{v}\leq\frac{G\left(K\right)}{v},
\]
and after rearranging, since $\ell>2$,
\[
e^{-\ell K}\leq\frac{G\left(K\right)}{v}\frac{e^{-2K}}{\ell}\leq\frac{G\left(K\right)}{v}.
\]
Thus 
\[
\ell\geq\frac{1}{K}\log\frac{v}{G\left(K\right)}.
\]
Putting this into the right hand side of (\ref{eq:technical_middle_of_keller_kindler})
gives
\begin{align*}
\int_{0}^{r}e^{-2s}G\left(s\right)ds & \geq v\left(1-e^{-r\ell}\right)\\
 & \geq v\left(1-e^{-\frac{r}{K}\log\frac{v}{G\left(K\right)}}\right)=v\left(1-\left(\frac{G\left(K\right)}{v}\right)^{\frac{r}{K}}\right).
\end{align*}

\end{proof}
\begin{proof}[Proof of Theorem \ref{thm:improved_keller_kindler}]
By Corollary \ref{cor:variance_from_sum_of_jumps}, 
\[
\var\left(f\right)=\var\left(f_{1}\right)=2\e\sum_{i=1}^{n}\int_{0}^{1}t\left(\partial_{i}f_{t}\right)^{2}dt=2\int_{0}^{1}t\cdot R\left(t\right)dt,
\]
and by change of variables this becomes 
\[
\var\left(f\right)=2\int_{0}^{\infty}e^{-2s}G\left(s\right)ds.
\]
Define $v=\int_{0}^{K}e^{-2s}G\left(s\right)ds$, where $K$ is the
constant from equation (\ref{eq:bound_on_R_at_constant_place}). Note
that since $K\geq1$ and $G$ is decreasing,
\begin{equation}
\var\left(f\right)-v=\int_{K}^{\infty}e^{-2s}G\left(s\right)ds\leq G\left(K\right)\stackrel{\left(\ref{eq:bound_on_R_at_constant_place}\right)}{\leq}L\cdot R\left(0\right)^{5/6}.\label{eq:bounding_v_and_stuff}
\end{equation}
Rearranging, this gives 
\begin{equation}
\frac{v}{G\left(K\right)}\geq\frac{\var\left(f\right)-LR\left(0\right)^{5/6}}{LR\left(0\right)^{5/6}}.\label{eq:bounding_v_ratio}
\end{equation}
Set $g\left(x\right)=\frac{1+f\left(x\right)}{2}$, and assume without
loss of generality that $\e f=f\left(0\right)\leq0$, so that $\e g=g\left(0\right)\leq\frac{1}{2}$
(if $f\left(0\right)>0$, we can take $g\left(x\right)=\left(1-f\left(x\right)\right)/2$).
This implies that 
\[
\var\left(g\right)=g\left(0\right)\left(1-g\left(0\right)\right)\geq\frac{1}{2}g\left(0\right).
\]
Invoking Lemma \ref{lem:biased_level_1_inequality} with $g$ and
$t=0$, there exists a constant $C$ such that
\begin{align}
R\left(0\right) & =\e\sum_{i=1}^{n}\left(\partial_{i}f\left(0\right)\right)^{2}=\norm{\grad f\left(0\right)}_{2}^{2}=4\norm{\grad g\left(0\right)}_{2}^{2}\nonumber \\
 & \leq Cg\left(0\right)^{2}\log\frac{e}{g\left(0\right)}\nonumber \\
 & \leq C'\left(\var\left(g\right)\right)^{2}\log\frac{C'}{\var\left(g\right)}\leq C''\var\left(f\right)^{2}\log\frac{4C''}{\var\left(f\right)}.\label{eq:temp_bounding_r0}
\end{align}
By (\ref{eq:can_assume_variance_is_small}), we can assume that $\var\left(f\right)$
is small enough so (\ref{eq:temp_bounding_r0}) implies
\begin{equation}
R\left(0\right)^{2/3}\leq\var\left(f\right).\label{eq:bounding_R_with_var}
\end{equation}
Plugging this into (\ref{eq:bounding_v_ratio}), we get
\[
\frac{v}{G\left(K\right)}\geq\frac{R\left(0\right)^{2/3}-LR\left(0\right)^{5/6}}{LR\left(0\right)^{5/6}}.
\]
For small enough $R\left(0\right)$, we have $LR\left(0\right)^{5/6}\leq\frac{1}{2}R\left(0\right)^{2/3}$,
and so 
\begin{equation}
\frac{v}{G\left(K\right)}\geq\frac{1}{2L}R\left(0\right)^{-1/6}.\label{eq:bounding_v_ratio_final}
\end{equation}
By Lemma \ref{lem:inequality_for_log_convex} and equations (\ref{eq:bounding_v_and_stuff})
and (\ref{eq:bounding_v_ratio}), we have
\begin{align}
\int_{0}^{r}e^{-2s}G\left(s\right)ds & \geq v\left(1-\left(\frac{G\left(K\right)}{v}\right)^{r/K}\right)\nonumber \\
 & \geq\left(\var\left(f\right)-LR\left(0\right)^{5/6}\right)\left(1-\left(2L\cdot R\left(0\right)^{1/6}\right)^{r/K}\right).\label{eq:kk_one_last_step}
\end{align}
This allows us to prove (\ref{eq:again_keller_kindler_nice}):
\begin{align*}
S_{\eps}\left(f\right) & =\e\sum_{i=1}^{n}\int_{0}^{\sqrt{1-\eps}}t\left(\partial_{i}f_{t}\right)^{2}dt\\
 & =\int_{0}^{\sqrt{1-\eps}}tR\left(t\right)dt\\
 & \leq\int_{0}^{e^{-\eps/2}}tR\left(t\right)dt\\
 & =\var\left(f\right)-\int_{0}^{-\eps/2}e^{-2s}G\left(s\right)ds\\
 & \stackrel{\left(\ref{eq:kk_one_last_step}\right)}{\leq}LR\left(0\right)^{5/6}+\var\left(f\right)\left(2L\cdot R\left(0\right)^{1/6}\right)^{\eps/2K}\\
 & \stackrel{\left(\ref{eq:bounding_R_with_var}\right)}{\leq}L\var\left(f\right)R\left(0\right)^{1/6}+\var\left(f\right)\left(2L\cdot R\left(0\right)^{1/6}\right)^{\eps/2K}\\
 & \leq C\cdot\var\left(f\right)R\left(0\right)^{\eps/\left(12K\right)}
\end{align*}
for some universal constant $C$.
\end{proof}
\bibliographystyle{amsalpha}
\bibliography{boolean_concentration_via_pathwise_stochastic_analysis}

\appendix

\section{Appendix 1: $p$-biased analysis}

For $p=\left(p_{1},\ldots,p_{n}\right)\in\left[0,1\right]^{n}$, let
$\mu_{p}$ be the measure 
\begin{align*}
\mu_{p}\left(y\right) & =\prod_{i=1}^{n}\frac{1+y_{i}\left(2p_{i}-1\right)}{2}=w_{\left(2p-1\right)}\left(y\right),
\end{align*}
which sets the $i$-th bit to $1$ with probability $p_{i}$. Let
\begin{align}
\omega_{i}\left(y\right) & =\frac{1}{2}\left(\frac{1-2p_{i}}{\sqrt{p_{i}\left(1-p_{i}\right)}}+y_{i}\frac{1}{\sqrt{p_{i}\left(1-p_{i}\right)}}\right),\label{eq:p_biased_characteristic_function}
\end{align}
and for a set $S\subseteq\left[n\right]$, define $\omega_{S}\left(y\right)=\prod_{i\in S}\omega_{i}\left(y\right)$.
Then every function $f$ can be written as 
\begin{align}
f\left(y\right) & =\sum_{S\subseteq\left[n\right]}\hat{f}_{p}\left(S\right)\omega_{S}\left(y\right)\nonumber \\
 & :=\sum_{S\subseteq\left[n\right]}\left(\e_{\mu_{p}}\left[f\cdot\omega_{S}\right]\right)\omega_{S}\left(y\right)\\
 & =\sum_{S\subseteq\left[n\right]}\left(\sum_{y\in\left\{ -1,1\right\} ^{n}}f\left(y\right)\omega_{S}\left(y\right)w_{2p-1}\left(y\right)\right)\omega_{S}\left(y\right)\label{eq:p_biased_fourier_decomposition}
\end{align}
The coefficients $\hat{f}_{p}:=\e_{\mu_{p}}\left[f\cdot\omega_{S}\right]$
are called the ``$p$-biased'' Fourier coefficients of $f$.

The $p$-biased influence of the $i$-th bit is 
\[
\Inf_{i}^{p}\left(f\right)=4p_{i}\left(1-p_{i}\right)\p_{y\sim\mu_{p}}\left[f\left(x\right)\neq f\left(x^{\oplus i}\right)\right].
\]
If $f$ is monotone, then 
\begin{equation}
\Inf_{i}^{p}\left(f\right)=2\sqrt{p_{i}}\sqrt{1-p_{i}}\hat{f}_{p}\left(\left\{ i\right\} \right).\label{eq:influence_for_biased_monotone_functions}
\end{equation}
The $p$-biased Fourier coefficients are related to the derivatives
of $f$ by the following proposition, whose proof (using slightly
different notation) can be found in \cite[Section 8]{odonnell_analysis_of_boolean_functions}.
\begin{prop}
\label{prop:derivatives_and_biased_fourier_coefficients}Let $S=\left\{ i_{1},\ldots,i_{k}\right\} \subseteq\left[n\right]$
be a set of indices, $x\in\left(-1,1\right)^{n}$, and $p=\frac{1+x}{2}$.
Then 
\[
\partial_{i_{1}}\ldots\partial_{i_{k}}f\left(x\right)=\left(\prod_{i\in S}\frac{4}{\sqrt{1-x_{i}^{2}}}\right)\hat{f}_{p}\left(S\right).
\]
\end{prop}

\section{Appendix 2: Postponed proofs}
\begin{proof}[Proof of Lemma \ref{lem:biased_level_1_inequality}]
The lemma is similar to \cite[Proposition 2.2]{talagrand_how_much_are_increasing_sets_positively_correlated},
but applied to a biased product distribution rather than to the uniform
distribution on $\left\{ -1,1\right\} ^{n}$. For completeness, we
present here a general proof, which does not explicitly use hypercontractivity.

Using (\ref{eq:harmonic_extension_is_weighted_sum}), the harmonic
extension $g\left(x\right)$ may be written as
\[
g\left(x\right)=\sum_{y}w_{x}\left(y\right)g\left(y\right),
\]
where $w_{x}\left(y\right)=\prod_{i}\left(1+x_{i}y_{i}\right)/2$.
Since differentiation and harmonic extensions commute, we have
\begin{align*}
\partial_{i}g\left(x\right) & =\frac{\partial}{\partial x_{i}}\sum_{y}w_{x}\left(y\right)g\left(y\right)\\
 & =\sum_{y}w_{x}\left(y\right)y_{i}\frac{g\left(y\right)}{1+x_{i}y_{i}}\\
 & =\frac{1}{1-t^{2}}\sum_{y}w_{x}\left(y\right)y_{i}g\left(y\right)\left(1-x_{i}y_{i}\right)\\
 & =\frac{1}{1-t^{2}}\nu\left(A\right)\left(\frac{\int_{A}y_{i}d\nu}{\nu\left(A\right)}-x_{i}\right),
\end{align*}
where $\nu$ is the harmonic measure $w_{x}\left(y\right)$, and $A=\mathrm{supp}\left(g\right)$.
Under this notation, 
\[
g\left(x\right)=\nu\left(A\right)\text{ and }\partial_{i}g\left(x\right)=\int_{A}\frac{y_{i}}{1+x_{i}y_{i}}d\nu.
\]
Let $\left\{ \alpha_{i}\right\} _{i=1}^{n}$ be numbers such that
$\sum_{i=1}^{n}\alpha_{i}^{2}=1$, and let $h:\left\{ -1,1\right\} ^{n}\to\r$
be defined as
\[
h\left(y\right)=\sum_{i=1}^{n}\alpha_{i}\frac{y_{i}}{1+x_{i}y_{i}}.
\]
Let $Y\in\left\{ -1,1\right\} ^{n}$ have distribution $\nu$. Recall
that the sub-gaussian norm of a random variable $R\in\r$ is defined
as $\norm R_{\psi_{2}}=\inf\left\{ s>0\mid\e\exp\left(R^{2}/s^{2}\right)\leq2\right\} $,
while the sub-gaussian norm of a random vector $R\in\r^{n}$ is defined
as $\norm R_{\psi_{2}}=\sup_{r\in S^{n-1}}\norm{\left\langle R,r\right\rangle }_{\psi_{2}}$
(see e.g \cite[Sections 2.5 and 3.4]{vershynin_high_dimensional_probability}
for more about sub-gaussian norms). The random variable $\frac{Y_{i}}{1+x_{i}Y_{i}}$
is bounded in magnitude by $\left(1-t\right)^{-1}$, and so has sub-gaussian
norm bounded by $C\left(1-t\right)^{-1}$ for some constant $C$.
By \cite[Lemma 3.4.2]{vershynin_high_dimensional_probability}, a
random vector $Z$ with independent, mean-zero sub-gaussian entries
is also sub-gaussian, with $\norm Z_{\psi_{2}}\leq C\max_{i}\norm{Z_{i}}_{\psi_{2}}$.
Thus the random vector $\left(\frac{Y_{1}}{1+x_{1}Y_{1}},\ldots,\frac{Y_{n}}{1+x_{n}Y_{n}}\right)$
has sub-gaussian norm bounded by $C\left(1-t\right)^{-1}$ as well,
which means that for every $s>0$
\[
\p\left[\abs{h\left(Y\right)}\geq s\right]\leq2\exp\left(-Cs^{2}\left(1-t\right)^{2}\right).
\]
Let $s_{0}\geq\frac{1}{\sqrt{C}}$. Then 
\begin{align*}
\int_{A}\abs hd\nu & =\int_{0}^{\infty}\nu\left(\left\{ \abs h\geq t\right\} \intersect A\right)d\nu\\
 & \leq\int_{0}^{\infty}\min\left(\nu\left(A\right),2e^{-Cs^{2}\left(1-t\right)^{2}}\right)ds\\
 & \leq\nu\left(A\right)s_{0}+\frac{2}{C\left(1-t\right)^{2}}\int_{s_{0}}^{\infty}\frac{s}{\sqrt{C}}C\left(1-t\right)^{2}e^{-Cs^{2}\left(1-t\right)^{2}}ds\\
 & \leq\nu\left(A\right)s_{0}+\frac{2}{C^{3/2}\left(1-t\right)^{2}}e^{-Cs_{0}^{2}\left(1-t\right)^{2}}.
\end{align*}
Taking $s_{0}=\sqrt{\frac{1}{C\left(1-t\right)^{2}}\log\frac{e}{\nu\left(A\right)}}\geq\frac{1}{\sqrt{C}}$
gives
\begin{align*}
\int_{A}\abs hd\nu & \leq\nu\left(A\right)\frac{1}{\sqrt{C}\left(1-t\right)}\sqrt{\log\frac{e}{\nu\left(A\right)}}+\frac{1}{C\left(1-t\right)^{2}}e^{-\log\frac{e}{\nu\left(A\right)}}\\
 & \leq\frac{L}{\left(1-t\right)^{2}}\nu\left(A\right)\sqrt{\log\frac{e}{\nu\left(A\right)}}
\end{align*}
for some $L>0$. In particular, 
\begin{equation}
\int_{A}hd\nu\leq\frac{L}{\left(1-t\right)^{2}}\nu\left(A\right)\sqrt{\log\frac{e}{\nu\left(A\right)}}\label{eq:bound_on_gradient_almost_finished}
\end{equation}
as well. Now choose $\alpha_{i}=\partial_{i}g\left(x\right)\left(\sum_{i=1}^{n}\partial_{i}g\left(x\right)^{2}\right)^{-1/2}$,
and observe that 
\begin{align*}
\left(\int_{A}hd\nu\right)^{2} & =\left(\int_{A}\sum_{i=1}^{n}\frac{\partial_{i}g\left(x\right)}{\sqrt{\sum_{i=1}^{n}\partial_{i}g\left(x\right)^{2}}}\frac{y_{i}}{1+x_{i}y_{i}}d\nu\right)^{2}\\
 & =\frac{1}{\sum_{i=1}^{n}\partial_{i}g\left(x\right)^{2}}\left(\sum_{i=1}^{n}\partial_{i}g\left(x\right)\int_{A}\frac{y_{i}}{1+x_{i}y_{i}}d\nu\right)^{2}\\
 & =\frac{1}{\sum_{i=1}^{n}\partial_{i}g\left(x\right)^{2}}\left(\sum_{i=1}^{n}\partial_{i}g\left(x\right)^{2}d\nu\right)^{2}=\sum_{i=1}^{n}\partial_{i}g\left(x\right)^{2}=\norm{\grad g}_{2}^{2}.
\end{align*}
Together with (\ref{eq:bound_on_gradient_almost_finished}), this
gives the desired result.
\end{proof}
\begin{proof}[Proof of Lemma \ref{lem:biased_level_2_inequality}]
Using Proposition \ref{prop:derivatives_and_biased_fourier_coefficients}
for $S=\left\{ i,j\right\} $ and the fact that $\abs{x_{i}}=t$,
\begin{align}
\norm{\nabla^{2}g\left(x\right)}_{HS}^{2} & =\sum_{i,j=1}^{n}\left(\partial_{i}\partial_{j}g\left(x\right)\right)^{2}\nonumber \\
 & =\sum_{i,j=1}^{n}\left(\frac{16}{1-t^{2}}\hat{g}_{p}\left(\left\{ i,j\right\} \right)\right)^{2}\nonumber \\
 & \leq2C\left(t\right)\sum_{S\subseteq\left[n\right],\abs S=2}\hat{g}_{p}\left(S\right)^{2}.\label{eq:level_2_hilbert_schmidt_1}
\end{align}
The following lemma, which bounds the sum of squares of $p$-biased
Fourier coefficients, is immediately obtained from \cite[Lemma 6]{keller_kindler_quantitative_noise_sensitivity}.
\begin{lem}
Let $0\leq t<1$ and let $p\in\left(0,1\right)^{n}$ be such that
$p_{i}\in\left\{ \frac{1+t}{2},\frac{1-t}{2}\right\} $ for all $i$.
For a function $g:\left\{ -1,1\right\} ^{n}\to\left\{ -1,1\right\} $,
let 
\[
\mathcal{W}\left(f\right)=p\left(1-p\right)\sum_{i=1}^{n}\Inf_{i}^{p}\left(g\right)^{2}.
\]
There exists a function $C\left(t\right)$ such that for every $g$,
\begin{equation}
\sum_{S\subseteq\left[n\right],\abs S=2}\hat{g}_{p}\left(S\right)^{2}\leq C\left(t\right)\mathcal{W}\left(g\right)\cdot\log\left(\frac{2}{\mathcal{W}\left(g\right)}\right).\label{eq:quantitative_noise_sensitivity_simplified}
\end{equation}
\end{lem}

Combining (\ref{eq:level_2_hilbert_schmidt_1}) and (\ref{eq:quantitative_noise_sensitivity_simplified}),
we get 
\begin{equation}
\norm{\nabla^{2}g\left(x\right)}_{HS}^{2}\leq C\left(t\right)\mathcal{W}\left(g\right)\cdot\log\left(\frac{2}{\mathcal{W}\left(g\right)}\right).\label{eq:level_2_hilbert_schmidt_2}
\end{equation}
As stated in equation (\ref{eq:influence_for_biased_monotone_functions}),
for monotone functions the influence of the $i$-th bit is given by
\[
\Inf_{i}^{p}\left(g\right)=2\sqrt{p_{i}}\sqrt{1-p_{i}}\hat{g}_{p}\left(\left\{ i\right\} \right),
\]
and so 
\[
\mathcal{W}\left(f\right)=p\left(1-p\right)\sum_{i=1}^{n}\Inf_{i}^{p}\left(g\right)^{2}=4p^{2}\left(1-p\right)^{2}\sum_{i=1}^{n}\hat{g}_{p}\left(\left\{ i\right\} \right)^{2}.
\]
On the other hand, using Proposition \ref{prop:derivatives_and_biased_fourier_coefficients}
with $S=\left\{ i\right\} $, 
\begin{align*}
\norm{\grad g\left(x\right)}_{2}^{2} & =\sum_{i=1}^{n}\left(\partial_{i}g\left(x\right)\right)^{2}\\
 & =\frac{16}{1-t^{2}}\sum_{i=1}^{n}\hat{g}_{p}\left(\left\{ i\right\} \right)^{2}\\
 & =\frac{4}{\left(1-t^{2}\right)p^{2}\left(1-p\right)^{2}}\mathcal{W}\left(g\right):=C'\left(t\right)\mathcal{W}\left(g\right).
\end{align*}
Plugging this into (\ref{eq:level_2_hilbert_schmidt_2}), we see that
for some $C\left(t\right)$ we have
\[
\norm{\nabla^{2}g\left(x\right)}_{HS}^{2}\leq C\left(t\right)\norm{\grad g\left(x\right)}_{2}^{2}\log\left(\frac{C\left(t\right)}{\norm{\grad g\left(x\right)}_{2}^{2}}\right).
\]
\end{proof}
\begin{proof}[Proof of Lemma \ref{lem:sum_of_jumps_to_integral}]
 Assume first that $g_{t}$ satisfies condition (\ref{enu:sum_of_jumps_left}),
meaning that it is left-continuous and measurable with respect to
the filtration generated by $\left\{ B_{s}\right\} _{0\leq s<t}$.
To prove (\ref{eq:sum_of_jumps_to_integral_main}), assume first that
$t_{1}>0$, so that the number of jumps that $B_{t}$ makes in the
time interval $\left[t_{1,}t_{2}\right]$ is almost surely finite.
For any integer $N>0$, partition the interval $\left[t_{1,}t_{2}\right]$
into $N$ equal parts, setting $t_{k}^{N}=t_{1}+\frac{k}{N}\left(t_{2}-t_{1}\right)$
for $k=0,\ldots,N$. Since almost surely none of the jumps of $B_{t}$
occur at any $t_{k}^{N}$, and since $g_{t}$ is left-continuous,
we almost surely have
\[
\sum_{t\in J_{i}\intersect\left[t_{1,}t_{2}\right]}4t^{2}g_{t}=\lim_{N\to\infty}\sum_{k=0}^{N-1}4\left(t_{k}^{N}\right)^{2}g_{t_{k}^{N}}\one_{J_{i}\intersect\left[t_{k}^{N},t_{k+1}^{N}\right]\neq\emptyset}.
\]
Since $g_{t_{k}^{N}}$ is bounded, the expression
\[
\sum_{k=0}^{N-1}4\left(t_{k}^{N}\right)^{2}g_{t_{k}^{N}}\one_{J_{i}\intersect\left[t_{k}^{N},t_{k+1}^{N}\right]\neq\emptyset}
\]
is bounded in absolute value by a constant times the number of jumps
of $B_{t}$ in the interval $\left[t_{1,}t_{2}\right]$, which is
integrable. By the dominated convergence theorem, we then have

\begin{align*}
\e\sum_{t\in J_{i}\intersect\left[t_{1,}t_{2}\right]}4t^{2}g_{t} & =\lim_{N\to\infty}\e\sum_{k=0}^{N-1}4\left(t_{k}^{N}\right)^{2}g_{t_{k}^{N}}\one_{J_{i}\intersect\left[t_{k}^{N},t_{k+1}^{N}\right]\neq\emptyset}.
\end{align*}
Since $g_{t_{k}^{N}}$ is measurable with respect to $\left\{ B_{s}\right\} _{0\leq s<t_{k}^{N}}$,
it is independent of whether or not a jump occurred in the interval
$\left[t_{k}^{N},t_{k+1}^{N}\right]$, and the expectation breaks
up into
\begin{equation}
\e\sum_{t\in J_{i}\intersect\left[t_{1,}t_{2}\right]}4t^{2}g_{t}=\lim_{N\to\infty}\sum_{k=0}^{N-1}\e\left[4\left(t_{k}^{N}\right)^{2}g_{t_{k}^{N}}\right]\e\left[\one_{J_{i}\intersect\left[t_{k}^{N},t_{k+1}^{N}\right]\neq\emptyset}\right].\label{eq:sum_to_integral}
\end{equation}
The set $J_{i}=\mathrm{Jump}\left(B_{t}^{\left(i\right)}\right)$
is a Poisson process with rate $1/2t$, and so the number of jumps
in the interval $\left[t_{k}^{N},t_{k+1}^{N}\right]$ distributes
as $\mathrm{Pois}\left(\lambda\right)$, where 
\[
\lambda=\int_{t_{k}^{N}}^{t_{k+1}^{N}}\frac{1}{2t}dt=\frac{1}{2}\log\frac{t_{k+1}^{N}}{t_{k}^{N}}.
\]
The probability of having at least one jump is then equal to
\begin{align*}
\p\left[J_{i}\intersect\left[t_{k}^{N},t_{k+1}^{N}\right]\neq\emptyset\right] & =1-e^{-\lambda}=1-\sqrt{\frac{t_{k}^{N}}{t_{k+1}^{N}}}=1-\sqrt{1-\frac{\left(t_{2}-t_{1}\right)/N}{t_{k+1}^{N}}}\\
 & =\frac{\left(t_{2}-t_{1}\right)/N}{2t_{k+1}^{N}}+O\left(\frac{1}{N^{2}}\right).
\end{align*}
Plugging this into display (\ref{eq:sum_to_integral}), we get
\[
\e\sum_{t\in J_{i}\intersect\left[t_{1,}t_{2}\right]}4t^{2}g_{t}=\lim_{N\to\infty}\e\sum_{k=0}^{N-1}4\left(t_{k}^{N}\right)^{2}g_{t_{k}^{N}}\left(\frac{\left(t_{2}-t_{1}\right)/N}{2t_{k}^{N}}+O\left(\frac{1}{N^{2}}\right)\right).
\]
The factor $O\left(\frac{1}{N^{2}}\right)$ is negligible in the limit
$N\to\infty$, since the sum contains only $N$ bounded terms. We
are left with 
\begin{align*}
\lim_{N\to\infty}\e\sum_{k=0}^{N-1}4\left(t_{k}^{N}\right)^{2}g_{t_{k}^{N}}\frac{\left(t_{2}-t_{1}\right)/N}{2t_{k+1}^{N}} & =\lim_{N\to\infty}\e\sum_{k=0}^{N-1}\left[2t_{k}^{N}g_{t_{k}^{N}}\right]\frac{t_{2}-t_{1}}{N}\\
\left(\text{bounded convergence}\right) & =\e\lim_{N\to\infty}\sum_{k=0}^{N-1}\left[2t_{k}^{N}g_{t_{k}^{N}}\right]\frac{t_{2}-t_{1}}{N}.
\end{align*}
Since $g_{t}$ is continuous almost everywhere, by the definition
of the Riemann integral, the limit is equal to $2\e\int_{t_{1}}^{t_{2}}t\cdot g_{t}dt$,
and we get 

\[
\e\sum_{t\in J_{i}\intersect\left[t_{1,}t_{2}\right]}4t^{2}g_{t}=2\e\int_{t_{1}}^{t_{2}}t\cdot g_{t}dt
\]
for all $t_{1}>0$. Taking the limit $t_{1}\to0$ gives the desired
result for $t_{1}=0$ by continuity of the right hand side in $t_{1}$.

The proof for condition (\ref{enu:sum_of_jumps_right}), where $g_{t}=g\left(B_{t}\right)$
is similar: Since $g_{t}$ is right-continuous, we now approximate
the sum in the left hand side of (\ref{eq:sum_of_jumps_to_integral_main})
using the right-endpoint of the interval:
\[
\sum_{t\in J_{i}\intersect\left[t_{1,}t_{2}\right]}4t^{2}g_{t}=\lim_{N\to\infty}\sum_{k=0}^{N-1}4\left(t_{k+1}^{N}\right)^{2}g\left(B_{t_{k+1}^{N}}\right)\one_{J_{i}\intersect\left[t_{k}^{N},t_{k+1}^{N}\right]\neq\emptyset}.
\]
Using the same reasoning as above we may interchange the expectation
and limit, obtaining. 
\[
\e\sum_{t\in J_{i}\intersect\left[t_{1,}t_{2}\right]}4t^{2}g_{t}=\lim_{N\to\infty}\e\sum_{k=0}^{N-1}4\left(t_{k+1}^{N}\right)^{2}g\left(B_{t_{k+1}^{N}}\right)\one_{J_{i}\intersect\left[t_{k}^{N},t_{k+1}^{N}\right]\neq\emptyset}.
\]
Since $B_{t_{k+1}^{N}}$ is independent of whether or not a jump occurred
in the interval $\left[t_{k}^{N},t_{k+1}^{N}\right]$, the expectation
again breaks up into
\[
\e\sum_{t\in J_{i}\intersect\left[t_{1,}t_{2}\right]}4t^{2}g\left(B_{t}\right)=\lim_{N\to\infty}\sum_{k=0}^{N-1}\e\left[4\left(t_{k+1}^{N}\right)^{2}g\left(B_{t_{k+1}^{N}}\right)\right]\e\left[\one_{J_{i}\intersect\left[t_{k}^{N},t_{k+1}^{N}\right]\neq\emptyset}\right].
\]
From here onwards the proof is identical.
\end{proof}
\begin{proof}[Proof of Lemma \ref{lem:differential_inequality}]
Let $t_{1}=\inf\left\{ t\mid g\left(t\right)\geq K\right\} $, and
denote $L=\max\left\{ x\log\frac{K}{x}\mid x\in\left[0,K\right]\right\} $.
Note that $L$ depends only on $K$. Then for all $t\le t_{1}$, we
have
\[
g'\left(t\right)\leq C\cdot L.
\]
Integrating, this means that for all $t\leq t_{1}$
\[
g\left(t\right)\leq g\left(0\right)+tCL\leq\frac{K}{2}+tCL.
\]
In particular, $t_{1}\geq\frac{K}{2CL}$, otherwise we'd have $g\left(t_{1}\right)<K$,
contradicting the definition of $t_{1}$ and continuity of $g$. Denoting
$t_{0}=\frac{K}{4CL}$, we must have $g\left(t\right)<K$ for all
$t\in\left[0,t_{0}\right]$. This ensures that $\log\frac{K}{g\left(t\right)}$
is positive in this interval, which means we can rearrange the differential
inequality (\ref{eq:differential_inequality}) to give
\[
-\frac{g'\left(t\right)}{g\left(t\right)\log\frac{g\left(t\right)}{K}}\leq C
\]
for all $t\in\left[0,t_{0}\right]$. A short calculation reveals that
the left hand side is the derivative of $-\log\log\left(K/g\right)$.
Integrating from $0$ to $t$, we get 
\[
\log\log\frac{K}{g\left(0\right)}-\log\log\frac{K}{g\left(t\right)}\leq Ct.
\]
Rearranging gives
\[
\log\log\frac{K}{g\left(t\right)}\geq\log\log\frac{K}{g\left(0\right)}-Ct,
\]
and exponentiating twice gives the desired result. 
\end{proof}

\end{document}